\documentclass{amsart}
\usepackage{amssymb,latexsym,amscd}

\numberwithin{equation}{section}

\newtheorem{theorem}{Theorem}[section]
\newtheorem{proposition}[theorem]{Proposition}

\newtheorem{corollary}[theorem]{Corollary}
\newtheorem{remark}[theorem]{Remark}
\newtheorem{lemma}[theorem]{Lemma}
\newtheorem{example}[theorem]{Example}
\newtheorem{definition}[theorem]{Definition}
\newtheorem{problem}[theorem]{Problem}

\newtheorem{maintheorem}[theorem]{Main Theorem}

\newtheorem{question}[theorem]{Question}

\def\FF{\mathcal{F}}
\def\R{\mathbb{R}}
\def\ZZ{\mathbb{Z}}
\def\QQ{\mathbb{Q}}
\def\CC{\mathbb{C}}
\def\C{\mathcal{C}}

\def\UU{\mathcal{U}}

\def\1{\mathbb{1}}

\def\aa{\mathfrak{a}}

\def\ll{\mathfrak{l}}

\def\gg{\mathfrak{g}}
\def\mm{\mathfrak{m}}
\def\nn{\mathfrak{n}}
\def\hh{\mathfrak{h}}
\def\VV{\mathcal{V}}

\def\bb{\mathfrak{b}}

\def\pp{\mathfrak{p}}

\def\R{\mathcal R}
\def\OO{\mathcal{O}}
 \def\zz{\mathfrak{z}}

\begin{document}

\title{Equivariant Quantized Symmetric Algebras}
\author{Sebastian Zwicknagl}
\maketitle

\begin{abstract}
Let $\gg$ be a Lie bialgebra and let $V$ be a finite-dimensional $\gg$-module. We study deformations of the symmetric algebra of $V$ which are equivariant with respect to an action of the quantized enveloping algebra $U_h(\gg)$, resp. $U_q(\gg)$ . We investigate, in particular, such quantizations obtained from the quantization of certain Lie bialgebra structures on the semidirect product of $\gg$ and $V$. We classify these structure in the important special case, when $\gg$ is complex, simple, with quasitriangular Lie bialgebra structure and $V$ is a simple g-module. We then introduce more a general notion, co-Poisson module algebras and their quantizations, to further address the problem and show that many known examples of quantized symmetric algebras can be described in this language.
\end{abstract}

\tableofcontents

 \section{Introduction}

The main objective of this paper is to study the following problem and provide a framework for its   solution.

\begin{problem}
\label{pr:equiv-defs}
Let $\gg$ be a Lie bialgebra over a field $k$ of characteristic $0$ and let $V$ be a finite-dimensional $\gg$-module.  Find all $U_h(\gg)$-equivariant (flat) deformations of the symmetric algebra $S(V)$.
\end{problem}
The problem is, apparently, non-trivial and interesting.  A. Berenstein and the author constructed in \cite{BZ} a natural quantum symmetric algebra, namely the braided symmetric algebra, for a finite dimensional module of the quantized enveloping algebra $U_q(\gg)$ of a reductive complex Lie algebra $\gg$. These algebras are (not necessarily flat) deformations of the corresponding classical symmetric algebras, and we showed that, if $\gg=sl_2$ and the modules under consideration are simple, then  the deformations are flat if and only if the dimension of the module was less than four.    Indeed, the deformation of the two-dimensional module is isomorphic to the (two-dimensional) quantum plane, while Vancliff \cite{Van} and Rossi-Doria  \cite{R-D} had previously studied the braided symmetric algebra of the four-dimensional simple $U_q(sl_2)$-module and shown that it was not a flat deformation. Vancliff, in particular, uncovered interesting non-commutative geometry associated with this algebra.

 In order to approach Problem \ref{pr:equiv-defs} we shall consider the following two-fold problem, which deals with one of the most interesting cases of Problem \ref{pr:equiv-defs}. Recall that  the semidirect product $\gg\ltimes V$ has a natural {\it inhomogeneous} Lie algebra structure with Lie bracket defined by
$$[g+v,g'+v']=[g,g']+g.v'-g'.v\ ,$$
 where for all $g,g'\in\gg$ and $v,v'\in V$. Indeed, $V$ is an Abelian Lie ideal in $\gg\ltimes V$ and there exists a semidirect factorization $U(\gg\ltimes V)=U(\gg)\otimes S(V)$. Additionally, recall that the universal enveloping algebra $U(\gg)$ of any Lie bialgebra $\gg$ can be quantized (see e.g. \cite{EK} and \cite{ES}).

\begin{problem}
\label{pr:equiv-defs-2}
  Let $\gg$ be a Lie bialgebra. For each $\gg$-module $V$:

(a) Find all Lie bialgebra structures on the semidirect product $\gg\ltimes V$ that are compatible with the Lie bialgebra structure on $\gg$. We refer to them as semidirect Lie bialgebra structures.

(b) Find all  quantized enveloping algebras $U_h(\gg\ltimes V)$ of $U(\gg\ltimes V)$  that admit a semidirect factorization into $U_h(\gg)$ and $S_h(V)$.
\end{problem}

We prove in Theorem \ref{th:qia} that if $\gg$ is semisimple and  $\gg\ltimes V$ is a semidirect Lie bialgebra then $U_h(\gg\ltimes V)$  admits a semidirect factorization into $U_h(\gg)$ and $S_h(V)$. Hence, if $\gg$ is semisimple, it suffices to solve Problem \ref{pr:equiv-defs-2}(a) in order to completely solve Problem \ref{pr:equiv-defs-2}. The following main result of the paper constitutes the solution to Problem \ref{pr:equiv-defs-2}(a) in the  important case when $\gg$ is a complex simple Lie algebra with the standard Lie bialgebra structure and $V$ a simple $\gg$-module.

\begin{maintheorem}(Theorem \ref{th:main})
\label{th:main-intro}
Let $\gg$ be a simple complex Lie algebra with the standard bialgebra structure and let $V$ be a nontrivial simple $\gg$-module.  Then there exists a central extension $\gg'\cong \gg\oplus \zz$ of $\gg$ such that  $\gg'\ltimes V$ admits a  semidirect   Lie bialgebra structure, if and only if the pair $(\gg,V)$ is  one of the following:

  (i) $(sl_n(\CC),V)$ where $V\in\{V,V^*,S^2V,S^2V^*,\Lambda^2 V,\Lambda^2 V^*\}$, where $V=\CC^n$ is the first fundamental $sl_n(\CC)$-module,

 (ii) the defining module of  $(so(n)$  or the spin modules  for $(so(10)$.

(iv)  the minuscule modules for $E_6$.
\end{maintheorem}

Note that these modules are exactly the geometrically decomposable modules listed by Howe \cite[ch.4]{Ho} and Stembridge \cite[Remark 2.3]{Stem}.

 Moreover, we prove in  Theorem \ref{th:main} that the same classification result also holds for all quasitriagular Lie bialgebra structures given by  Belavin-Drinfeld triples introduced in (\cite{BD}).

  Surprisingly, the classification in Theorem \ref{th:main-intro} almost coincides with the classification results of Theorems 1.1 and 1.2 in the author's earlier paper \cite{ZW}.  It turns out that if $\gg\ltimes V$ admits a semidirect  Lie bialgebra structure, then the classical $r$-matrix obtained from the Belavin-Drinfeld triple defines a Poisson bracket on the symmetric algebra $S(V)$. The corresponding Poisson structures were indeed classified in \cite{ZW}, and their symplectic foliations were studied by Goodearl and Yakimov in \cite{GY} in the case of the standard $r$-matrix.

 Theorem \ref{th:main-intro} has interesting connections to other areas of Lie theory, namely the pairs listed in Theorem \ref{th:main-intro}(i)-(iv) appear as geometrically decomposable modules in Classical Invariant Theory as studied by Howe (see e.g. Howe \cite{Ho}), and in the classification of the Hermitian Symmetric Spaces (see e.g. Howe \cite{Ho}). Moreover, the semidirect products  $(\gg\oplus \CC)\ltimes V$ can be interpreted as maximal parabolic subalgebras with Abelian radicals $\nn=V$ inside complex simple Lie algebras--thus establishing the connection to the minuscule Grassmannians.

Moreover, the quantizations of these symmetric algebras encompass many well-known quantized coordinate rings. Among them are, corresponding to the modules in Theorem \ref{th:main-intro}(i), the well known quantum planes, the quantum symmetric matrices introduced by  Noumi in \cite{Nou} and the quantum anti-symmetric matrices introduced by Strickland  in \cite{Str}. Similarly, one obtains quantum Euclidean space, introduced by   Faddeev, Reshitikhin and Takhtadzhyan \cite{RTF},   as the quantization of the symmetric algebras $(so(n),\CC^n)$ in Theorem \ref{th:main-intro}(ii).  Additionally, all of these quantized symmetric algebras can be interpreted as braided symmetric algebras; i.e., quantum analogs of symmetric algebras introduced by A. Berenstein and the author in \cite{BZ}.  If $\gg$ is quasitriangular, then the quantized symmetric algebras obtained as semidirect factorizations  of $U_h(\gg\ltimes V)$, can be interpreted as symmetric algebras in the associated co-boundary categories of  $U_h(\gg)$-modules, as is shown in Section \ref{se:qsa-coboundary}.

However, there are examples of equivariant deformations of symmetric algebras which do not correspond to semidirect Lie bialgebras, for instance the quantized symmetric algebra $S_h(sl_n)$ of the adjoint $sl_n$-module introduced by Donin in \cite{Don}.  To address these examples as well, we introduce the notion of {\it co-Poisson module algebras} in Section \ref{se: and co-Poisson}. A  co-Poisson module algebra is a pair $(H,A)$ of a co-Poisson Hopf algebra $H$ and a cocommutative $H$-module bialgebra $A$ together with a map $\delta:A\to H\otimes A\oplus A\otimes H$ such that $\delta$ satisfies the co-skew-symmetry, co-Leibniz rule and the co-Jacobi identity. We then construct, in Section \ref{se:co-dec-locadfin} a large family of  co-Poisson module algebras, based on Joseph and Letzter's work \cite{JL} on the $ad$-finite part of the quantized enveloping algebras $U_q(\gg)$ of a complex semisimple Lie algebra $\gg$ and results of Lyubashenko and Sudbery \cite{Lyu-Sud} on quantum Lie algebras.

 We can now reformulate and generalize Problems 1.1 and 1.2 as follows:

\begin{problem}
\label{pr:co-poisson}
\noindent(a) Classify co-Poisson module algebras $(H,A)$.\\
 \noindent(b) Classify quantizations of co-Poisson module algebras $(H_h,A_h)=(H\ltimes A)_h$.
 .
\end{problem}

As we explained above,  we completely solve the Problem  \ref{pr:co-poisson} in the case when $H$ is the universal enveloping algebra of a quasitriangular complex simple Lie bialgebra and $A$ the symmetric algebra of a finite dimensional $H$-module.  Moreover, we show in Section \ref{se:co-dec-locadfin} that the quantum symmetric algebras of the adjoint $U_h(sl_n)$-module defined by Donin in \cite{Don} can be constructed as the associated graded of a quantized co-Poisson module algebra $(H,A)$ where $H=U(sl_n)$ and $A=U(sl_n)$. We plan to address Problem \ref{pr:co-poisson} in more generality in subsequent papers.

The paper will be organized as follows: In Section \ref{se: and co-Poisson}, we introduce the notion of a  co-Poisson Hopf algebra and its quantization.  In Section \ref{se:LBAsand quantization} we recall the notions  and properties of finite-dimensional Lie bialgebras and introduce  semidirect Lie bialgebras. We then classify  in Section \ref{se:quasitriangular} semidirect Lie bialgebras $\gg\ltimes V$ where $\gg$ is a simple Lie algebra with quasitriangular Lie bialgebra structure and $V$ a simple $\gg$-module (Theorem \ref{th:main}). Our proof relies on the author's previous results in \cite{ZW}. Moreover, if $\gg$ has the standard Lie bialgebra structure we give a direct proof for the classification theorem, using results of Hodges and Yakimov in \cite{HY} and \cite{HY1} about the double of a Lie bialgebra.   The following section investigates the quantizations of the  co-Poisson module algebras arising from  finite-dimensional submodules of $U_q(\gg)$ and use them  to construct a quantum symmetric algebra for the adjoint $U_q(sl_n)$-module which coincides with Donin's construction in \cite{Don}. In the appendices we will recall some  well known facts about the quantization of Lie bialgebras, the quantized universal enveloping algebras $U_q(\gg)$ and the (semi)classical limit.

 {\bf Acknowledgments}  The author would like to thank Arkady Berenstein, Ken Goodearl and Milen Yakimov for interesting and stimulating discussions.

  \section{Co-Poisson module algebras}
\label{se: and co-Poisson}
  \subsection{Hopf algebras and their modules}
   In this section we will  recall the definitions of algebras, coalgebras, bialgebras and Hopf algebra, as well as their modules and co-modules.
    We first recall the  definitions of  monoidal and braided monoidal categories.

  \begin{definition}
  A  monoidal category is a category $\C$ with a functor $ \bigotimes: \C \times\C \to \C $ that associates an object $X\otimes Y$ to each pair $(X,Y)$ of objects, and a morphism $f\otimes g$ to each pair $(f,g)$ of morphisms, and an object $1$
 such that for $X \in Ob(\C )$ one has $ 1\otimes X\cong X\otimes 1\cong X$, and such that the pentagonal axiom  $X,Y,Z, W\in Ob(\C)$ is satisfied; i.e., one has
 $$((X\otimes Y)\otimes Z)\otimes W\cong (X\otimes (Y\otimes Z))\otimes W \cong (X\otimes Y)\otimes (Z\otimes W)\cong$$
 $$ X\otimes ((Y\otimes Z)\otimes W)\cong X\otimes (Y\otimes (Z\otimes W))\ .$$

  \end{definition}

    \begin{definition}
   \noindent(a) Denote by $\tau: (X,Y)\to (Y,X)$ the permutation of factors in $\C\times \C$. A braided  monoidal category $(\C,\R)$ is a  monoidal category $\C$ with a natural transformation $\R$ between the functors $\bigotimes:  C\times C\to C$ and $\bigotimes\circ\tau:  C\times C\to C$ satisfying  the following relations:

    \begin{equation}
 \label{eq: triangle right}
 \R_{X, Y\otimes Z}= (Id_{Y}\otimes\R_{X,Z}) \circ (\R_{X,Y}\otimes Id_{Z})\ ,
 \end{equation}
 \begin{equation}
 \label{eq: triangle left}
  \R_{X\otimes Y,Z}=(\R_{X,Z}\otimes Id_{Y}  ) \circ (Id_{X}\otimes\R_{Y,Z} )\ .\end{equation}

 When $A$ and $B$ are fixed we may sometimes abbreviate $\R_{A,B}=\R$.

 \noindent (b) If the braiding $\R$ satisfies additionally   $\R_{B,A}\circ\R_{A,B}=Id_{A\otimes B}$\ for all objects $A$ and $B$, then we  refer to $(\C,\R)$ as  a symmetric category.
  \end{definition}
\begin{example}
The permutation of factors defines a symmetric braiding on the category of vectorspaces over any field.
\end{example}

 An associative {\it unital algebra}   in a monoidal  category $\C$
is an object $A$ of $\C$ with a map
$\mu:  A\otimes A\to A$ called multiplication, and a map $\eta: 1\to A$, called unit, satisfying the following relations:

$$
 \begin{CD}
A\otimes A \otimes A@>Id\otimes \mu>>A\otimes A  \\
@ V\mu\otimes Id VV @VV\mu V\\
A \otimes A@>\mu>>  A
\end{CD}$$
 \vspace{3mm}
 $$
 \begin{CD} 1\otimes A  @>\eta \otimes Id>>A\otimes A  \\
@ V VV @VV\mu V\\
A @>Id >>  A
\end{CD}  \quad\quad
 \begin{CD}
A\otimes 1  @>Id \otimes \eta>>A\otimes A  \\
@ V VV @VV\mu V\\
A @>Id >>  A\ .
\end{CD}
$$
We will abbreviate $\mu(a\otimes b)= a\cdot b$, and denote the category of algebras in $\C$ by $Alg(\C)$.  A (left)-$A$-{\it module} is an object $V$ of $\C$  with a map $m: A\otimes V\to V$, called the action of $A$ on $V$ satisfying the following:

$$
 \begin{CD}
A\otimes A \otimes V@>Id\otimes m>>A\otimes V  \\
@ V\mu\otimes Id VV @VV m V\\
A \otimes V@>\mu>>  V\ .
\end{CD}
  $$

A  {\it co-unital coalgebra}  is an object $B$ of $\C$  with a map $\Delta: B\to B\otimes B$, called the co-multiplication and a map $\varepsilon: B\to 1$, the co-unit, satisfying the following relations:

$$
 \begin{CD}
B @>\Delta >>B\otimes B\\
@ V\Delta VV @VV Id\otimes\Delta  V\\
B \otimes B@>\Delta\otimes Id>>  B\otimes B\otimes B
\end{CD}
$$ \vspace{3mm}
$$
 \begin{CD}
B @>\Delta >>B\otimes B  \\
@ V VV @VV Id\otimes \varepsilon V\\
B\otimes 1 @>Id >> B\otimes 1
\end{CD}  \quad\quad
 \begin{CD}
B @>\Delta>>B\otimes B  \\
@ V VV @VV\varepsilon\otimes Id V\\
1\otimes B @>Id >>  1\otimes B\ .
\end{CD}
$$

A left-$A$-{\it comodule structure} on an  object $V$ of $\C $ is  a linear map $\delta:V\to A\otimes V$, called the co-action of $A$ on $V$ satisfying
$$
 \begin{CD}
V@>\delta >>A\otimes V\\
@ V\delta VV @VV Id\otimes\delta  V\\
A\otimes A@>\Delta\otimes Id>>  A\otimes A\otimes V\ .
\end{CD}
$$
 The category of left $A$-modules in $\C$ consists of the left $A$-modules as objects and structure preserving maps as morphisms. A left $A$ module algebra, resp. coalgebra  is an algebra,  resp. coalgebra in the category of $A$-modules.

The category of left $A$-comodules in $\C$ consists of the left $A$-comodules as objects and structure preserving maps as morphisms. A left-$A$-comodule algebra, resp. coalgebra is an algebra, resp.  coalgebra $V$ in the category of $A$-modules.

A {\it bialgebra} is an object $A$  of $Alg(\C)$ which has an algebra and a coalgebra structure such that the co-multiplication is a homomorphism of algebras: $\Delta: A\to A\otimes A$.  We can define the notion of $A$-module or co-module bialgebras analogous to the case of algebras and coalgebras.

A Hopf algebra over $k$ is a bialgebra $H$  together with an algebra anti-automorphism $S: H\to H$ , called the antipode, satisfying the following relation:
$$
  \begin{CD}
A@>\Delta >>A\otimes A\\
@ V\eta\circ\varepsilon  VV @V S\otimes Id VvvV\\
A  @<\mu<<  A\otimes A
\end{CD} \quad\quad
 \begin{CD}
A@>\Delta >>A\otimes A\\
@ VV\eta\circ\varepsilon  V @VV Id\otimes S V\\
A  @<\mu<<  A\otimes A\ .
\end{CD}
  $$

We have so far defined left module and comodule structures. Note that {\it right} module and comodule structures can be defined analogously. The following fact regarding the structure of the categories of algebras, coalgebras, bialgebras and Hopf algebras is well known.

\begin{lemma}
\label{le:tensor-prod-alg}
Let $\C$ be a braided monoidal category. The categories of algebras, coalgebras, bialgebras and Hopf algebras have a natural tensor structure defined by

$$\mu_{A\otimes B} =(\mu_A\otimes \mu_B)\circ \sigma_{23}\ ,$$
$$\Delta_{A\otimes B}=\sigma_{23}\circ(\Delta_A\otimes \Delta_B)\ ,$$
where $\sigma_{23}=Id\otimes \sigma\otimes Id$ denotes the braiding acting on the second and third factors.
\end{lemma}

We have the following well known fact.

\begin{lemma}
\label{le:right-action}
Let $H$ be a Hopf algebra and $V$ a left $H$-module. Then we can define a right action of $H$ on $V$ via $v.h=S(h).v$.
\end{lemma}

 We say that a Hopf algebra is cocommutative if it satisfies
 $$
  \begin{CD}
A@>\Delta >>A\otimes A\\
@ V\Delta  VV @V \tau VV\\
A\otimes A  @<Id_{A\otimes A}<<  A\otimes A\ ,
\end{CD}
  $$
 where $\tau$ denotes the permutation of factors $\tau(a\otimes b)=b\otimes a$.

 Now we are ready to state an important fact.
 \begin{proposition}
\noindent(a) If $H$ is a  Hopf algebra and $A$ a  $H$-left module bialgebra, then $H\otimes A$ has a natural structure of a bialgebra with multiplication $h\cdot a=h(a)$ for all $h\in h$ and $a\in A$ and comultiplication

 $$\Delta(h\cdot a)=\Delta(h)\cdot\Delta(a)= h_{(1)}\cdot a_{(1)}\otimes h_{(2)}\cdot a_{(2)}\ .$$

 for all $a\in A$ and $h\in H$, with braiding $\sigma_{23}$ given by the permutation of factors.

 \noindent(b) If $A$ is a $H$-left module Hopf algebra, then the extension of the antipodes defines an Hopf algebra structure on $H\otimes A$.
 \end{proposition}

  \begin{proof}
  We have to show that the bialgebra structure is well defined. Indeed, we have $a\cdot h= h_{(1)} \cdot S(h_{(2)}(a)$, using the right action of $H$ on $A$ given in Lemma  \ref{le:right-action}.  Part (a) is follows.

 Part (b) is easily verified. The proposition is proved.
  \end{proof}

  \subsection{Co-decorated and Co-Poisson-module algebras}

 In this section we will introduce the notions of co-decorated and co-Poisson module algebras. In order to state our results in the most efficient way we need the following notation: Let $A$ and $B$ be two vectorspaces over the field $k$. Then  we denote by $A\wedge B\subset A\otimes B\oplus B\otimes A$, the subspace spanned by skewsymmetric elements. We now make the following definitions.

 \begin{definition}
 \label{def:co-decoratedstuff}
 \noindent(a) A co-decorated bialgebra is a pair $(B,\delta)$ of a cocommutative bialgebra $B$ and a map $\delta:B\to B\wedge B$  satisfying  the co-Leibniz rule

 \begin{equation}
 \label{co:Leibniz1}
 (1\otimes \Delta)\circ \delta=(\delta\otimes 1)\circ \Delta+\sigma_{23}\circ(\delta\otimes 1)\circ \Delta\ ,
 \end{equation}
 and the compatibility condition
 \begin{equation}
 \label{eq:co-Leibniz}
 \delta(a\cdot b)=\delta(a)\Delta(b)+\Delta(a)\delta(b)\ .
 \end{equation}

 \noindent(b) A co-decorated bialgebra $(B,\delta)$ satisfying the co-Jacobi identity
 \begin{equation}
 \label{eq:co-Jacobi}
 (Cyc)\circ (1\otimes \delta)\circ \delta =0\ ,
  \end{equation}

 where $(Cyc)$ denotes the sum over the  cyclic permutations,  is  called a co-Poisson algebra.
 \end{definition}

  Now let $(H,\delta_H)$ be a cocommutative co-decorated bialgebra. The category of co-decorated $H$-module bialgebras consists of  pairs $(A,\delta)$ where $A$ is a unital  cocommutative $H$-module bialgebra and $\delta: H\otimes A\to (H\otimes A)^{\otimes 2}$ is a co-decoration, satisfying  $\delta(1\otimes a)\in ((H\otimes 1)\wedge (1\otimes A))$  and $\delta(h\otimes1)=\sigma_{23}\circ(\delta_H(h)\otimes (1_A)^{\otimes 2})$.
  The morphisms in the category are structure-preserving maps.

\begin{definition}
\label{def:co-Poisson-codec}
A  co-Poisson module algebra is a co-decorated $H$-module bialgebra  $(H,A)$ such that the co-decoration $\delta$ on $H\otimes A$ satisfies the co-Jacobi-identity \eqref{eq:co-Jacobi}.
\end{definition}

Note that this implies that $H$ is a co-Poisson bialgebra, and that it immediately invites the following question.

 \begin{problem}
 \label{pr:co-Poisson}
 Classify all co-Poisson and co-decorated structures associated to certain classes of bialgebras and module algebras, such as the enveloping algebras of complex semisimple Lie algebras with quasitriangular Lie-bialgebra structure.
 \end{problem}

 \subsection{Quantization of co-Poisson co-decorated Bialgebras}

In this section  we will introduce the quantization problem for the category of co-Poisson co-decorated Hopf algebras. We recall the notions of a quantization of a bialgebra in the Appendix Section \ref{se:quant.of LBA}.  We need the following definition.

 \begin{definition}
 A quantization of a co-Poisson module algebra $(H,A)$ is a pair $(H_h,A_h)$, where $H_h$ and $A_h$ are a Hopf, resp. bialgebra over $k[[h]]$ together with a bialgebra structure on $H_h\otimes A_h$ such that   $H_h\otimes A_h$ is a quantization of the co-Poisson module algebra structure on $(H,A)$.
 \end{definition}

We will discuss several classes of quantizations of co-decorated co-Poisson Hopf algebras in the following sections, which leads us to the following question.

\begin{problem}
Classify those co-Poisson Hopf algebras which can be obtained from co-Poisson co-decorated structures and which admit a  quantization. Are the quantizations unique?
\end{problem}

 \section{Lie bialgebras}
\label{se:LBAsand quantization}
Recall the definition of a Lie bialgebra over a field $k$ of characteristic zero.
 \begin{definition}
 A Lie bialgebra is a triple $(\gg, [\cdot,\cdot], \delta)$ of a vector space $\gg$ with Lie bracket $[\cdot,\cdot]$ and a map $\delta:\gg\to \gg\wedge\gg$ such that
 \begin{enumerate}
 \item $\delta$ defines a Lie  bracket on $\gg^*$.
 \item $\delta$ and $[\cdot,\cdot]$ are compatible via
 \begin{equation}
 \label{eq:cocycle}
 \delta([a,b])=[\delta(a), b\otimes 1+1\otimes b]+[a\otimes 1+1\otimes a,\delta(b)]\ .
 \end{equation}
 \end{enumerate}
 \end{definition}
 Note that if $(\gg, [\cdot,\cdot], \delta)$  is a Lie bialgebra, then so is $(\gg^*, \delta^*, , [\cdot,\cdot]^*)$, with Lie bracket $\delta^*$ and cobracket $[\cdot,\cdot]^*$.

Lie bialgebras are interesting objects in relation to our discussion of co-Poisson structures because of the following well-known result.

\begin{proposition}
\label{pr: bialg-co-Poisson}
Let $(\gg,[\cdot,\cdot],\delta)$ be a Lie bialgebra. The universal enveloping algebra $U(\gg)$ admits a co-Poisson structure defined by $\delta$ on $\gg\in U(\gg)$ and extended to $U(\gg)$ by \eqref{eq:co-Leibniz}.\end{proposition}

The notions of a classical $r$-matrix and a quasitriangular Lie algebra, which we will introduce next, is very important for the theory of quantizations.
\begin{definition}
A classical $r$-matrix is an element  $r\in\gg\otimes \gg$ such that $r+r^{op}$ is $\gg$-invariant and $r\in\gg\otimes \gg\subset U(\gg)\otimes U(\gg)$ satisfies the Classical Yang-Baxter Equation (CYBE)
\begin{equation}
\label{eq:CYBE}
[r_{12}, r_{13}]+[r_{12},r_{23}]+[r_{13},r_{23}]=0\in U(\gg)^{\otimes 3}\ ,
\end{equation}
where $r_{12}=r\otimes 1$, $r_{23}=1\otimes r$ and $r_{13}=r_{(1)}\otimes 1\otimes r$.
\end{definition}

\begin{definition}
\noindent (a) A Lie bialgebra $(\gg,[\cdot,\cdot],\delta)$ is called quasitriangular if there exists a classical $r$-matrix $r\in \gg\otimes \gg$ such that for all $g\in\gg$
$$\delta(g)=[r,1\otimes g+g\otimes 1]\ .$$

\noindent (b) $(\gg,[\cdot,\cdot],\delta)$ is called triangular if there exists a skew-symmetric classical $r$-matrix $r\in \Lambda^2 \gg$ such that for all $g\in\gg$
$$\delta(g)=[r,1\otimes g+g\otimes 1]\ .$$

\noindent (c) A Lie bialgebra $(\gg,[\cdot\cdot],\delta)$ is called coboundary, if there exists $r\in \Lambda^2 \gg$ such that for all $g\in\gg$
$$\delta(g)=[r,1\otimes g+g\otimes 1]\ .$$
\end{definition}

Conversely, every classical $r$-matrix defines a Lie bialgebra.

\begin{proposition}(see e.g. \cite{Ch-P})
Let $\gg$ be a Lie algebra and let $r\in\gg\otimes \gg$ be a solution of the CYBE such that $r+r^{op}$ is $\gg$-invariant. Then $\delta:\gg\to \gg\otimes \gg$ given by
$$\delta(x)=[r,x\otimes 1+1\otimes x]$$
defines a Lie bialgebra structure on $\gg$.
\end{proposition}

\subsection{The Double of a Lie bialgebra}

Let $\aa$ be a Lie algebra and let $\aa^*$ be a Lie algebra structure on the dual Lie of $\aa$. Define a new algebra
$D(\aa)$, the Drinfeld double,  such that $D(\aa)=\aa\oplus \aa^*$ as vectorspaces and endowed with a skewsymmetric bracket
$[\cdot,\cdot]:D(\aa)\otimes D(\aa)\to D(\aa)$ such that its restrictions to  $\aa$ and $\aa^*$ are given by the Lie brackets on $\aa$ and $\aa^*$, respectively, and such that

$$[x,\xi]=ad^*_\aa(x)(\xi)-ad^*_{\aa^*}(\xi)(x)\ , x\in\aa, \xi\in\aa^*\ .$$
Here $ad^*_\aa$ and $ad^*_{\aa^*}$ denotes the coadjoint action of $\aa$ on $\aa^*$ and $\aa^*$ on $\aa$, respectively.
Analogously, we define the Lie algebra structure on $D(\gg^*)$.
 The following fact is well known.

 \begin{proposition}(see e.g. \cite{Ch-P} or \cite[ch.4.1]{ES})
  The algebra $D(\aa)$ is a Lie algebra, i.e. the bracket satisfies the Jacobi identity if and only if $\aa^*$ defines a Lie bialgebra structure on $\aa$.
 \end{proposition}

We then have the following fact originally due to Drinfeld.

\begin{proposition}\cite[Theorem 4.1]{ES}
Let $(\aa,[\cdot,\cdot],\delta)$ be a finite dimensional Lie bialgebra. Then the double $D(\aa)$ is a quasitriangular Lie bialgebra.
\end{proposition}

\subsection{Manin triples}
In this section we will introduce the notion of a finite-dimensional Manin triple which is closely related to the double of a Lie bialgebra, again following \cite[ch. 4]{ES} closely.

\begin{definition}
A triple of finite dimensional Lie algebras $(\gg,\gg_{+},\gg_-)$ such that
\begin{enumerate}
\item $\gg_+$ and $\gg_-$ are Lie subalgebras of $\gg$ and such that $\gg=\gg_+\oplus \gg_-$ as a vector space, and
\item $\gg_+$ and $\gg_-$ are isotropic subalgebras with respect to a nondegenerate invariant bilinear form  $\langle \cdot,\cdot\rangle:\gg\otimes\gg\to \gg$
\end{enumerate}
is called a Manin triple.
\end{definition}

 The form    $\langle \cdot,\cdot\rangle$ induces a nondegenerate pairing $\gg_+\otimes \gg_-\to \CC$ and hence a Lie algebra isomorphism $\gg_-\cong \gg_+^*$. Therefore $\gg_-$ defines a Lie coalgebra structure $\delta:\gg_+\to\gg_+\wedge\gg_+$.  Denote by $[\cdot,\cdot]$ the Lie bracket on $\gg_+$. The following fact is well known.

 \begin{proposition}
 \label{pr:q-tr-manin triples}
\noindent(a)  Let $(\gg,\gg_+,\gg_-)$ be a finite-dimensional Manin triple. Then, $(\gg_+,[\cdot,\cdot],\delta)$ is a Lie bialgebra. Moreover, $\gg$ is isomorphic as a Lie algebra to $D(\gg_+)$.

\noindent(b) Let $(\gg,[\cdot,\cdot],\delta)$ be a Lie bialgebra. Then $D(\gg)$ is a quasitriangular Lie bialgebra whose $r$-matrix is the canonical element corresponding to the trace on $\gg\otimes \gg^*$.
 \end{proposition}

\subsection{Semidirect Lie Bialgebras}
The following is our key definition.

  \begin{definition}
  \label{def:sdplba}
  A semidirect Lie bialgebra  is a pair $(\gg, V)$ of a complex  semisimple Lie algebra $\gg$ and a finite-dimensional $\gg$-module $V$ together with a Lie bialgebra   structure $\delta$ on $\gg\ltimes V$ such that    $\gg$ with the restricted co-bracket is a Lie subbialgebra of $\gg\ltimes V$ and that for the restriction $\delta|_V: V\to V\wedge \gg$.
  \end{definition}
  We will sometimes refer to $\gg\ltimes V$ as a semidirect Lie bialgebra. We have the following fact.

\begin{proposition}
A  semidirect Lie bialgebra $\gg\ltimes V$ defines a co-Poisson  algebra $(U(\gg),S(V),\delta)$.
\end{proposition}

  \begin{proof}
  The assertion follows directly from Proposition \ref{pr: bialg-co-Poisson} and Definition \ref{def:co-Poisson-codec}.
  \end{proof}

We now obtain a first classification result.

\begin{proposition}
\label{pr:inh-triangular}
Let $\gg$ be a Lie algebra and let $V$ be a finite dimensional $\gg$-module. Let $r\in \Lambda^2 \gg$ be a skew symmetric solution of CYBE. Then $\delta:\gg\ltimes V\to\Lambda^2(\gg\ltimes V)$ given by
$$\delta(x)=[r,x\otimes 1+1\otimes x]$$
for all $x\in\gg\ltimes V$ defines a   semidirect  Lie bialgebra structure on $\gg\ltimes V$.
\end{proposition}

\begin{proof}
We establish first that  $\delta$ defines a Lie bialgebra structure.
Since $r$ is a solution of the CYBE it remains to verify that $r+r^{op}$ is $\gg\ltimes V$-invariant. But $r+r^{op}=0$ because we assumed $r$ to be skewsymmetric. It is now easy to see that $\gg$ is indeed a Lie subbialgebra of $\gg\ltimes V$. On the other hand $[r, 1\otimes x]\in \gg\otimes V$ and $[r_{(1)}, x] \otimes r_{(2)}=-[r_{(2)},x]\otimes r_{(1)}$ and hence $\delta(x)\in\gg\wedge V$. The proposition is proved.
\end{proof}

\begin{remark}
Proposition \ref{pr:inh-triangular} also follows directly from the well-known fact that if a Lie group $G$ acts on a manifold $M$, then any triangular Poisson Lie structure on $G$ defines a Poisson structure on $M$, compatible with the action of $G$.
\end{remark}

\section{Quasitriangular Structures and  Semidirect  Lie bialgebras}
\label{se:quasitriangular}

 In this section we will analyze  semidirect Lie bialgebras associated to quasitriangular Lie bialgebras. We will show how they give rise to certain Poisson algebras and classify all  semidirect Lie bialgebras $\gg\ltimes V$, where $\gg$ is a complex simple Lie algebra with Lie bialgebra structure defined by a Belavin-Drinfeld triple and $V$ a finite-dimensional simple $\gg$-module.

 \subsection{Quasitriangular Structures and Poisson brackets}

The main goal of this section is to prove the following result.

\begin{theorem}
\label{th:sdp-poisson}
Let $(\gg,[\cdot,\cdot],\delta)$ be a quasitriangular Lie bialgebra with classical $r$-matrix $r$ and let $\gg\ltimes V$ be a semidirect Lie bialgebra.  Then $r$ defines a Poisson structure on $S(V)$ defined on the generators by
$$\{u,v\}=r^-(u\otimes v)\ ,$$
where $r^-=\frac{1}{2}(r-r^{op})$.
\end{theorem}
\begin{proof}
In order to prove the theorem we have to describe the co-bracket on $\gg\ltimes V$ in terms of the classical $r$-matrix $r$. In order to accomplish this consider the Drinfeld double $D(\gg)$ and the associated Manin triple $(D(\gg),\gg,\gg^*)$. The following fact is well known.

\begin{lemma}\cite[ch. 4.2]{ES}
Let $(\gg,\delta)$ and $(\gg^*,\delta^*)$ be a Lie bialgebra and its dual.  The embeddings of $\gg$ and $\gg^*$ define maps $\delta:D(\gg)\to D(\gg)\wedge D(\gg)$, resp. $\delta^*:D(\gg)\to D(\gg)\wedge D(\gg)$. The Lie bialgebra structure $\tilde\delta$ on the double $D(\gg)$ is given by
$$\tilde\delta=\delta-\delta^*\ .$$
Moreover, $\gg$ and $\gg^*$ are Lie subbialgebras of $D(\gg)$.
\end{lemma}

It is well known (see e.g. \cite[ch 2]{HY}) that $r$ defines two Lie bialgebra homomorphisms $r_{\pm}:\gg^*\to \gg$ in the following way for $\xi\in\gg^*$:
$$r_+(\xi)=(\xi\otimes Id)r\ ,\quad r_-(\xi)=(id\otimes \xi)r\ .$$
Denote the image of $r_+$ by $\gg_+$.
Consider the Lie algebra $\gg_+\ltimes V$ and its double $D(\gg_+\ltimes V)$. Since $\gg_+$ is a Lie subbialgebra of $\gg$, we obtain that $\gg_+^*$ is a Lie subalgebra of $\gg^*$ and hence that  $D(\gg_+\otimes V)\subset D(\gg\ltimes V)$, is a Lie subalgebra. The Lie bialgebra structure on $D(\gg_+\ltimes V)$ is given by definition for all $x\in\gg_+\ltimes V$:
\begin{equation}
\label{eq:coalg on g+V}
\delta(x)=\sum_i e_i\otimes ad^*e^i(x)+e^*_i\otimes  ad^*e^i(x)+\sum_j v_j\otimes ad^*v^j(x)+v^j\otimes ad^*v_j (x)\ .
\end{equation}
where $\{e_i:i=1\in[1,n]\}$ and $\{v_j: j\in[1,m]\}$ are bases of $\gg_+$ and $V$ respectively and the $e^i$ and $v^j$ denote the corresponding dual basis vectors. Using the fact that $\delta_+(v)\in\gg_+\wedge V$ we obtain that
$$\delta_+(v)=\sum_i e_i\otimes ad^*e^i(v)+e^*_i\otimes  ad^*e^i(v)-ad^*e^i(v)\otimes e_i-ad^*e_i(v)\otimes e^i\ ,$$
where $\delta_+$ denotes the cobracket on $\gg_+\ltimes V$.
Denote by $c_+=\sum_i e_i\otimes e^i$ the canonical element of $\gg_+\otimes\gg_+^*$, by $c_+^{op}\in \gg_+^*\otimes \gg_+ $ its opposite, and by $c=c_++c_+^{op}$. We obtain that
$$\delta_+(v)=[c,1\otimes v-v\otimes 1]\ .$$

In order to complete our proof we will have to use the fact that $ \delta_V $ satisfies the co-Jacobi identity, i.e.
\begin{equation}
\label{eq:co-Jacobi1}
 Alt\circ (\delta_V\otimes 1)\circ \delta_V=0\ ,
 \end{equation}
where $Alt$ denotes the sum over all cyclic permutations $Alt=1+\tau_{123}+\tau_{132}$. Using the notations $c_{12}=c\otimes 1$, $c_{23}=1\otimes c$ and $c_{13}=\tau_{23}\circ (c\otimes 1)\circ \tau_{23}$ and the fact that $c$ is symmetric we can now compute:
$$(\delta_V\otimes 1)\circ \delta_V=(\delta_V\otimes 1)\circ [c,v\otimes 1-1\otimes v]$$
$$=[c_{12}, [c_{13},v\otimes 1\otimes 1-1\otimes 1\otimes v]-[c_{23}, 1\otimes v\otimes 1-1\otimes 1\otimes v]] $$
$$=[c_{12}, [c_{13},v\otimes 1\otimes 1-1\otimes 1\otimes v- 1\otimes v\otimes 1]]-[c_{12}, [c_{23}, 1\otimes v\otimes 1-1\otimes 1\otimes v+v\otimes 1\otimes 1]]\ .$$
The cyclic permutations are given by:
$$\tau_{123}\circ (\delta_V\otimes 1)\circ \delta_V)=[c_{23},[c_{12},1\otimes v\otimes 1-v\otimes 1\otimes 1\pm 1\otimes 1\otimes v]]-[c_{23}, [c_{13}, 1\otimes 1\otimes v-v\otimes 1\otimes 1\pm 1\otimes v\otimes 1]]\ ,$$
$$\tau_{132}\circ(\delta_V\otimes 1)\circ \delta_V)=[c_{13},[c_{23},1\otimes 1\otimes v-1\otimes v\otimes 1\pm v\otimes 1\otimes 1]]-[c_{13}, [c_{12}, v\otimes 1\otimes 1-1\otimes v\otimes 1\pm 1\otimes 1\otimes v]]\ .$$

Note that by the Jacobi identity $[a,[b,c]]-[b,[[a,c]=[[a,b],c]$. Therefore we can collect terms and rewrite the co-Jacobi identity  \eqref{eq:co-Jacobi1} in the form
$$0=[[c_{12}, c_{13}],v\otimes 1\otimes 1-1\otimes 1\otimes v- 1\otimes v\otimes 1]+[[c_{23}, c_{12}], 1\otimes v\otimes 1-1\otimes 1\otimes v- v\otimes 1\otimes 1]$$
$$+[[c_{13},c_{23}], 1\otimes v\otimes 1-1\otimes 1\otimes v- v\otimes 1\otimes 1] \ .$$

 Denote for $\phi\in \gg\otimes \gg$ by $[[\phi,\phi]]$ the Yang-Baxter operator (or Schouten square)
 $$[[\phi,\phi]]=[\phi_{12},\phi_{13}]+[\phi_{12},\phi_{23}]+[\phi_{13},\phi_{23}]\ .$$
We need the following known fact.
\begin{lemma}
\label{le:r- to casimir}
Let $\gg$ be a quasitriangular Lie algebra with $r$-matrix $r$ and $c=\frac{1}{2}(r+r^{op})$ and $r=\frac{1}{2}(r-r^{op})$. Then
$$ [c_{12}, c_{23}]=[c_{23},c_{13}]=[c_{13},c_{12}]=[[r^-,r^-]] \in \Lambda^3 \gg\ .$$
\end{lemma}
\begin{proof}
A proof of the identity $[c_{12}, c_{23}]=[[r^-,r^-]]$ and that $[[r^-,r^-]]\in\Lambda^3\gg$ can be found, in a more general setup, in \cite{ZW}, and the remaining identities follow directly from the fact that $[c_{12},c_{23}]$ is invariant under cyclic permutations.
\end{proof}

Thus we obtain that for $v\in V$:
$$0=-[[c_{12},c_{23}],v\otimes 1\otimes 1+1\otimes v\otimes 1+1\otimes 1\otimes v]\ .$$
Now let $u\wedge v\wedge w\in\Lambda^3 V$. Denote $\Delta_{123}(v)=v\otimes 1\otimes 1+1\otimes v\otimes 1+1\otimes 1\otimes v$ and  define $AS(u,v)$ as
$$AS(u,v)=1\otimes u\otimes v-1\otimes v\otimes u+v\otimes 1\otimes u-u\otimes 1\otimes v+v\otimes u\otimes 1-u\otimes v\otimes 1\ .$$

 Since $[u',v']=0$ for all $u',v'\in V$  we have
$$[[c_{12},c_{23}], u\wedge v\wedge w]=[[[c_{12},c_{23}],\Delta_{123}(u)], AS(v,w)]=0\ .$$
Therefore, $[[c_{12},c_{23}],\Lambda^3 V]=0$ and $[[r^-,r^-]](\Lambda^3 V)=0$. The assertion of the theorem now follows from the following fact.

\begin{proposition}\cite[Theorem 2.21]{ZW}
Let $V$ be a vectorspace and let $\Phi^-$ be a skewsymmetric endomorphism of $V\otimes V$. Then the following are equivalent:

\noindent(a) the Schouten square satisfies $[[\Phi^-,\Phi^-]](\Lambda^3 V)=0$.

\noindent(b)the endomorphism $\Phi^-$ defines a Poisson bracket on the symmetric algebra $S(V)$ given on the generators $\{u,v\}=\Phi^-(u\otimes v)$.
\end{proposition}
The proposition implies that $c_+^-=\frac{1}{2}(c_+-c_+^{op})$ defines a Poisson bracket on $S(V)$. We obtain under the standard identification $\phi:\gg_-\to \gg_+^*$ that $1\otimes \phi^{-1}(c_+)=r$ and hence that  $\phi(c_+^-)=r^-$ defines a Poisson bracket on $S(V)$.
Theorem \ref{th:sdp-poisson} is proved.
\end{proof}

\begin{remark}  Theorem  \ref{th:sdp-poisson}  yields another proof for Proposition \ref{pr:inh-triangular}, since in this situation $c=\frac{1}{1}(r+r^{op})=0$.
\end{remark}
\subsection{Semidirect Lie Bialgebras and Simple Lie algebras}
Quasitriangular Lie bialgebra structures associated to simple complex Lie algebras were classified by Belavin and Drinfeld in \cite{BD} in terms of Belavin-Drinfeld triples.
We first recall the classification of such Lie bialgebras by Belavin and Drinfeld in \cite{BD}, here presented following Etingof and Schiffmann \cite[ch. 5.3]{ES}, where there are also several examples and proofs. Then we will give a classification result for  semidirect Lie Bialgebras arising from these quasitriangular structures.

 Let $\gg=\nn_-\oplus \hh\oplus \nn_+$ be a simple complex Lie algebra, with given triangular decomposition, and let $\langle \cdot.\cdot\rangle$ be an invariant bilinear form on $\gg$ such that the square of the length of a long root is $2$. Denote by  $P(\gg)$ the lattice of integral weights, by $P^+(\gg)$ the monoid of dominant weights.  Denote  by  $R(\gg)$ the root-system of $\gg$ and by $R^{\pm}(\gg)$ the set of positive, resp. negative roots. Denote by $(\cdot,\cdot)_{\hh}$ and $(\cdot,\cdot)_{\hh^*}$ the standard inner product on $\hh$ and $\hh^*$, which we identify via the inner product.  Denote by $R(\gg)\subset \hh^*$ the set of roots, by $R^+(\gg)$ (resp. $R^-(\gg)$) the set of positive (resp. negative) roots and  by $\Delta=\{\alpha_1,\ldots,\alpha_n\}$ the set of simple  roots.  Denote by $E_{\alpha}$ and $F_\alpha$  for $\alpha\in R^+(\gg)$ and $H_{\alpha}\subset \hh$, $\alpha\in \R^+(\gg)$ the standard generators of $\gg$ with the property that $[E_\alpha,F_{\alpha}]=H_\alpha=\check\alpha =2\frac{\alpha}{(\alpha,\alpha)}\in \hh\subset \gg$.

   Let $\Delta$ be the basis for $R(\gg)$ corresponding to the chosen  triangular decomposition, and denote by  $\omega_i$ for $\alpha_i\in\Delta$ the $i$-th fundamental weight.

 \begin{definition}
 A Belavin-Drinfeld triple is a triple $(\Delta_1,\Delta_2,\tau)$ where $\Delta_1,\Delta_2\subset \Delta$ and $\tau:\Delta_1\to\Delta_2$ such that
 \begin{enumerate}
 \item $\tau$ is a bijective map preserving the form $(\cdot,\cdot)$.
 \item for any $\delta\in \Delta_1$ there exists $n>0$ such that $\tau(\delta)\in \Delta_2\backslash \Delta_1$.
 \end{enumerate}
 \end{definition}

The isomorphism $\tau$ extends to isomorphisms $\tau:\ZZ\Delta_1\to\ZZ\Delta_2$ and hence extends as follows  to $\tau:\gg_{\Delta_1} \to \gg_{\Delta_2}$, where $ \gg_{\Delta_i}$ is the semisimple part of the Levi subalgebra  corresponding to $\Delta_i$: For each root $\alpha$ of $ \gg_{\Delta_1}$  define
$\tau(E_{\alpha})=E_{\tau(\alpha)}$. Note that the isomorphism is not unique and depends on our previous choice of root vectors.

 Note that the second property yields that $\tau^n(\alpha)\ne \alpha$ for all $n>0$ for $\alpha\in \ZZ\Delta_1$.
We can therefore define a partial order on the set of positive roots $R^+(\gg)$ by setting $\alpha<\beta$ if there exists $n>0$ such that $\tau^n(\alpha)=\beta$ for some $n>0$. Denote by $c\in S^2\gg$ the Casimir element of $\gg$ and by $c_0$ its $\hh$-part.  We can now state the Belavin-Drinfeld classification.

\begin{theorem}(Belavin-Drinfeld \cite{BD})
Let $\gg$ be a complex simple Lie algebra  with invariant nondegenerate form $\langle\cdot,\cdot\rangle$ and triangular decomposition $\gg=\nn^+\oplus \hh\oplus \nn^-$. Let $(\Delta_1,\Delta_2,\tau)$ be a Belavin-Drinfeld triple. Let $r_0\in \hh\otimes \hh$ satisfy

$$r_0+r_0^{21}=c_0$$
$$(\tau(\alpha)\otimes 1)+(1\otimes \alpha)r_0=0\ , \text{for $\alpha\in \Gamma_1\subset \hh^*$}\ .$$

Now define
$$ r=r_0+\sum_{\alpha\in R^+} F_\alpha\otimes E_\alpha +\sum_{\alpha,\beta\in R^+, \alpha<\beta} F_{\alpha}\otimes E_\beta\ .$$

Then

\begin{enumerate}
\item $r$ satisfies CYBE  \eqref{eq:CYBE} and $r+r^{21}=c$.
\item Any $r$ satisfying CYBE  \eqref{eq:CYBE}  and $r+r^{21}=c$ is of the above form for a suitable choice of triangular decomposition of $\gg$.
\end{enumerate}
\end{theorem}

We will call such Lie bialgebra structures Belavin-Drinfeld Lie bialgebras. Note that any quasitriangular Lie bialgebra structure on a simple Lie algebra $\gg$ with classical $r$-matrix $r$ which does not correspond to a Belavin Drinfeld triple must be triangular. Indeed,  because  the space of symmetric invariants $(S^2\gg)^{\gg}$ is one-dimensional, we obtain that $r+r^{21}=0$ and $r\in\Lambda^2\gg$.
The following theorem is our main result regarding  semidirect Lie bialgebras.

  \begin{theorem}
 \label{th:main}
  Let $\gg$ be a complex simple Lie algebra with a  Belavin Drinfeld Lie bialgebra  structure. Additionally, let  $V$  be a simple $\gg$-module. Then the following are equivalent:

  \noindent (a) There exists a Lie algebra $\gg'=\gg\oplus \zz$, where $\zz$ is finite-dimensional and central in $\gg'$ such that $\gg'\ltimes V$ admits a semidirect Lie bialgebra structure.

  \noindent(b) The module $V$ is geometrically decomposable in the sense of Howe (\cite{Ho}); i.e. it corresponds to an Abelian radical.

  \noindent(c) The pair $(\gg,V)$ is one of the following:

    (i) $(sl_n(\CC),V_\lambda)$ where $\lambda\in\{\omega_1,2\omega_1,\omega_2,\omega_{n-2},\omega_{n-1}, 2\omega_{n-1}\}$.

 (ii) $(so(n),V_{\omega_1})$,$(so(5),V_{\omega_2})$, $(so(8),V_{\omega_i})$, $i=3,4$ and     $ (so(10),V_{\omega_j})$, $j=4,5$.

 (iii) $(sp(4),V_{\omega_2})$.

(iv) $(E_6,V_{\omega_1})$ and $(E_6, V_{\omega_6})$.
  \end{theorem}

\begin{proof}
It follows directly from  the proof of Theorem \ref{th:sdp-poisson} that if $\gg\ltimes V$ is a semidirect Lie bialgebra then $r^{-}=\frac{1}{2}(r-r^{op})$ defines a Poisson bracket on $S(V)$ or equivalently $[c_{12},c_{23}](\Lambda^3 V)=0\subset S^3V$ by  Lemma \ref{le:r- to casimir}. However, all simple modules for complex simple Lie algebras with this property were classified  by the author in Theorem 1.1 of the paper \cite{ZW}. The only modules not appearing in the list of Theorem \ref{th:main} are the natural representation $V_{\omega_1}$ of the Lie algebras  $sp(2n)$. However, we prove further below in Proposition \ref{pr: sp 2n} that $\gg\ltimes V_{\omega_1}$ does not admit a semidirect Lie bialgebra structure.

 It remains to show that there exist associated Lie bialgebra structures for the modules listed in Theorem \ref{th:main}. Recall that a parabolic subalgebra $\pp$ of a semisimple Lie algebra $\gg$ splits as a semidirect $\pp\cong \ll\ltimes \nn$, where $\ll$ is the Levi subalgebra and $\nn$ the radical, a nilpotent Lie algebra. Moreover, $\ll\cong \ll'\oplus \zz$, where $\ll'$ is semisimple and $\zz$ is central.
 Recall that the Abelian radicals in simple Lie algebras are well known and Howe showed in \cite{Ho} that the radicals are isomorphic, as modules over the semisimple part of the Levi subalgebra to the simple geometrically decomposable modules listed in  Theorem \ref{th:main}.
The following fact now yields the assertion of the theorem.

  \begin{proposition}
  \label{pr:abelian rads-lie bi}
 Let $\gg$ be a simple Lie algebra with the Belavin-Drinfeld Lie bialgebra  structure and $\pp$ a parabolic  subalgebra  which splits as  $\ll\ltimes \nn$. Suppose that $\ll\cong \ll'\oplus\zz$, where $\ll'$ is semisimple and $\zz$  central.  Suppose that $\ll'$ is simple with a Belavin-Drinfeld Lie bialgebra structure and suppose that $\nn$ is an Abelian Lie algebra. Then there exists a Lie bialgebra structure on $\gg$ such that $\pp\cong (\ll'\oplus\zz)\ltimes \nn$ is a Lie subbialgebra of $\gg$, hence  a semidirect Lie bialgebra.
 \end{proposition}
 \begin{proof}
  The parabolic subalgebras with Abelian radical are well known, indeed the pairs $(\ll',\nn)$ correspond to the ones listed in Theorem \ref{th:main} (iii), where $\nn$ is interpreted as a module over $\ll'$. Denote by $i:D_{\ll'}\hookrightarrow D_\gg$ the embedding of the Dynkin diagram of $\ll'$ into the Dynkin diagram of $\gg$ corresponding to the embedding of the Levi $\ll'$ into $\gg$. It is now easy to observe that if $(\Gamma_1,\Gamma_2,\tau)$ is the Belavin-Drinfeld triple corresponding to the Lie bialgebra structure on $\ll'$, then $i(\Gamma_1,\Gamma_2,\tau)$ defines a Belavin-Drinfeld triple for $\gg$ which has the desired properties. The proposition is proved.
   \end{proof}

 Theorem \ref{th:main} now follows, as the equivalence of (b) and (c) is shown by Howe in \cite[ch. 5.5]{Ho}.
\end{proof}

Recall that all quasitriangular Lie bialgebra structures on simple Lie algebras are either of triangular or of Belavin Drinfeld type. Theorem  \ref{th:main} and Proposition \ref{pr:inh-triangular} have the following direct consequence.
\begin{corollary}
Let $\gg$ be a simple Lie algebra with quasitriangular Lie bialgebra structure, and let  $\gg'=\gg\oplus\zz$ as above and let $\gg\ltimes V$ be a semidirect Lie bialgebra such that $V$ is a simple $\gg$-module. Then the Lie bialgebra structure on $\gg$ is triangular and $V$ may be any $\gg$-module, or $(\gg,V)$ is one of the pairs listed in Theorem  \ref{th:main}.
\end{corollary}

\subsection{Semidirect Lie bialgebras and the standard structure}
In this section we will give a direct proof of Theorem \ref{th:main} for the case of the standard Lie bialgebra structure on a simple Lie algebra $\gg$.  It shows which role the doubles $D(\gg)$ and $D(\gg\ltimes V)$ play in the classification of semidirect Lie bialgebras.  By linking the proof in detail to steps in the proof of Theorem 1.1 of the author's \cite{ZW}, we shed some light on the interplay between the co-Poisson geometry in the present paper and the $r$-matrix Poisson structures discussed in \cite{ZW}.  The standard $r$-matrix, corresponding to the Belavin Drinfeld triple $(\emptyset, \emptyset, \tau)$ is given by

 $$r=\sum_{\alpha\in R^+(\gg)} \frac{<H_\alpha,H_\alpha>}{2} E_\alpha\otimes F_\alpha+ r_0 .$$

  The standard Lie bialgebra structure on a complex simple Lie algebra is defined as $\delta(x)=[r,x\otimes 1+1\otimes x]$. The following result is our main theorem.

  \begin{theorem}
 \label{th:main1}
  Let $\gg$ be a complex simple Lie algebra  with standard Lie bialgebra structure and let  $V$  be a simple $\gg$-module. Then the following are equivalent:

  \noindent (a) The pair $(\gg,V)$ admits a semidirect  Lie bialgebra.

  \noindent(b) The module $V$ is geometrically decomposable in the sense of Howe (\cite{Ho}); i.e. it corresponds to an Abelian radical.

  \noindent(c) The pair $(\gg,V)$ is one of the following:

    (i) $(sl_n(\CC),V_\lambda)$ where $\lambda\in\{\omega_1,2\omega_1,\omega_2,\omega_{n-2},\omega_{n-1}, 2\omega_{n-1}\}$.

 (ii) $(so(n),V_{\omega_1})$,$(so(5),V_{\omega_2})$, $(so(8),V_{\omega_i})$, $i=3,4$ and     $ (so(10),V_{\omega_j})$, $j=4,5$.

 (iii) $(sp(4),V_{\omega_2})$.

(iv) $(E_6,V_{\omega_1})$ and $(E_6, V_{\omega_6})$.
  \end{theorem}

\begin{proof}
As a first step towards proving Theorem \ref{th:main1} have to understand the representation theory of the Lie algebras $\gg^*$  and $D(\gg)$ defined by a Belavin-Drinfeld triple.

 \begin{proposition}
 \label{pr:double-cyclic}
 Let $\gg$ be a simple Lie algebra and let $(\Gamma_1,\Gamma_2,\tau)$ be a Belavin-Drinfeld triple. If $V$ is a simple $\gg$-module then $V^*$ is graded into weight spaces by the action of $\hh^*$, and  $V^*$ is a cyclic module with the module structure induced from the action of $\nn^+$ or $\nn^-$ on the highest or lowest weight vectors.
 \end{proposition}

 \begin{proof}
We will prove the proposition by investigating the action of the double $D(\gg)$ on $V$ and $V^*$. The following fact is obvious.
\begin{lemma}
The double $D(\gg)$ acts naturally on $V$ and $V^*$ by suitable restricting the adjoint action of $D(\gg\ltimes V)$ on itself.
\end{lemma}

 We will next use the following result of Hodges and Yakimov \cite{HY}.
 \begin{lemma}
 \noindent (a)\cite[ Corollary 7.1]{HY} The double $D(\gg)$ is isomorphic to $\gg\oplus\gg$ as a Lie algebra.

 \noindent (b) The Lie algebra $\gg^*$ embeds $i:\gg^*\hookrightarrow \gg\oplus \gg$ such that $\nn+\oplus \nn^-\subset \gg\oplus\gg$ is contained in the image of $i$. The Lie algebra $\gg$ embeds into $\gg\oplus \gg$ via the diagonal map.

 \end{lemma}

 \begin{proof}
 Part (b) of the lemma follows directly from \cite[Corollary 7.1]{HY}  and the definition of the subalgebra $\bb$ in Corollary 7.1 which is given in Section 5 of the paper \cite{HY}.
 \end{proof}

 The previous lemma implies that if $V$ is a simple $\gg$-module then it must be isomorphic to $V\otimes V_0$ or  $V_0\otimes V$ as a $\gg\oplus\gg$-module, where $V_0$ denotes the trivial $\gg$-module. Here we assume that the first copy of $\gg$ acts on the first tensor factor while the second copy acts on the second factor. Similarly, $V^*$ will be isomorphic to  $V^*\otimes V_0$ or  $V_0\otimes V^*$. This directly implies that $V^*$ is a cyclic $(\nn^+,\nn^-)$-module generated by the lowest weight vectors, in the case   $V\otimes V_0$, or the highest weight vectors, in the case $V_0\otimes V$. The
 proposition is proved.
 \end{proof}

Now we will discuss some facts about the Lie algebra structure on $\gg^*$.
 Denote  by $\nn^*_{\pm}$ the dual of $\nn^{\pm}$ and by $\hh^*$ the dual of $\hh$.
 \begin{lemma}
 The Lie algebra $\gg^*$ has a triangular decomposition $\gg^*=\nn^*_+\oplus\hh^*\oplus\nn^*_-$,
where $\nn^*_{\pm}$ are nilpotent Lie algebras and $\hh^*$ is a maximal commutative subalgebra.
 \end{lemma}
 \begin{proof}
A proof of this fact can be found for example in Yakimov's paper \cite[ch. 3.1]{Y}
 \end{proof}

 Now we are ready to consider finite dimensional representations of $\gg^*$. The following fact is obvious.

 \begin{lemma}
 Let $V^*$ be a finite dimensional $\gg^*$-module. Then $\gg^*$ is naturally graded into eigenspaces for $\hh^*$, the weightspaces.
 \end{lemma}
Now choose a basis $e_\alpha, f_\alpha, h_i$  for $\alpha\in R^+(\gg)$  dual to the standard basis.   Note that if $\gg\ltimes V$ is a semidirect Lie bialgebra, then $V$ becomes a $\gg^*$ module in the double $D(\gg\ltimes V)$. We have the following fact.

 \begin{lemma}
 Let $\gg$ be a quasitriangular simple Lie bialgebra and let $\gg\ltimes V$ be a semidirect Lie bialgebra  consider the standard embedding  of $V \subset D(\gg\ltimes V)$. Then we can find a bigrading of $V\subset D(\gg\ltimes V)$  in $\gg^*$ and $\gg$-weight spaces.
 \end{lemma}

 \begin{proof}
 Suppose that $v\in V$ is an $\gg$-weight vector of weight $\lambda$. Note that $[h,H]=0\in D(\gg)$ for all $h\in\hh^*$ and $H\in\hh$. Since $\hh$ and $\hh^*$ are commutative one has for a weight $H_{\alpha}$ and $h\in \hh^*$:

 $$(\lambda,\alpha) [h,v]=[h,[H_{\alpha},v]]=[[h,H], v]+[H_\alpha, [h,v]]=0+[H_{\alpha},[h,v]$$

 Hence the action of $\hh^*$ preserves the $\gg$-weight spaces. We show analogously that $\hh$ preserves the $\hh^*$ weight spaces. The assertion of the lemma is now immediate.
 \end{proof}

  Now we  prove that simple modules not in the list (c) do not admit semidirect Lie bialgebra structures. The equivalence of (b) and (c) is shown by Howe in \cite[ch. 5.5]{Ho}.

 First note the following fact.

 \begin{lemma}
 Let $\gg$ be a simple Li algebra with standard Lie bialgebra structure. Then
 $[e_i, f_j]=0$ for all $i,j$ (including $i=j$).
 \end{lemma}

 \begin{proof}
 The assertion follows by straightforward calculation. Note that in particular $\delta(\hh)=0$, hence $[e_i,f_i]=0$.
 \end{proof}

  Now we have to consider whether the gradings are compatible with the action of the $e_i$ and $f_i$.

  \begin{lemma}
  \label{le:weight-differences}
  Let $\gg\ltimes V$ be a semidirect Lie bialgebra with $\gg$ simple with standard Lie bialgebra structure. Let $v^*\in (V(\lambda))^*$, the $\lambda\in P^+(\gg)$-weight space. Then,

  $e_i(v^*)\in (V(\lambda+\alpha_i))^*$ and $f_i(v^*)\in (V(\lambda-\alpha_i))^*$.
  \end{lemma}

 \begin{proof}
 Let $w\in V(\mu)$ and let
 $$\delta(w)= H'\wedge w+\sum_{\alpha\in R^+(\gg)} \left(E_{\alpha} \wedge w_{\alpha}+  F_{\alpha} \wedge w'_{\alpha} \right)\ .$$
Suppose that $H_\mu\in \hh$. Noting that $\delta(H_\mu)=0$ we obtain from the cocycle identity \eqref{eq:cocycle} that
 $(\lambda,\mu)\delta(w)=\delta(H_\mu.w)=[\Delta(H_\mu), \delta(w)]$,
  and hence that
  $$(\lambda, \mu) E_{\alpha} \wedge w_{\alpha}=[H_\mu, E_\alpha] \wedge w_\alpha+E_\alpha \wedge [H_\mu, w_\alpha]\ .$$
  The assertion now follows immediately for $e_i$ and by a similar argument also for $f_i$.
  The lemma is proved.
 \end{proof}

 The following proposition will be a main tool for proving Theorem \ref{th:main1}.

 \begin{proposition}
 \label{pr:not too big}
   Let $\gg\ltimes V$ be a  semidirect Lie bialgebra with $\gg$ simple with standard Lie bialgebra structure.  If $V=V_\lambda$ is a simple selfdual $\gg$-module, then either $2\lambda\in R^+(\gg)$  or $2\lambda-\alpha_i\in R^+(\gg)$.
 \end{proposition}

 \begin{proof}
 We will prove the following more general, but more technical fact which includes the assertion of the proposition as a special case. This lemma will be useful to prove the classification result in the case of   modules which are not selfdual.

 \begin{lemma}
\label{le:nottoobig}
    Let $\gg\ltimes V$ be a semidirect  Lie bialgebra with $\gg$ simple with standard Lie bialgebra structure. Suppose that $V=V_\lambda$ is simple.  Let $\mu\in P(\gg)$ such that there exists a simple root  $\alpha_i\in R(\gg)$ such that $(\lambda,\alpha_i)>0$ and $(\mu,\alpha_i)<0$.   Then, $\lambda-\mu\in R^+(\gg)$  or $\lambda-\mu+\alpha_i\in R^+(\gg)$ .
 \end{lemma}

 \begin{proof}
 Since $V^*_\lambda$ is cyclic as a $\gg^*$-module (see Proposition \ref{pr:double-cyclic}) we may assume without loss of generality that $V_\lambda^*$ is generated by $v_\lambda^*$, the dual of a highest weight vector and that $V_\lambda^*=\nn^*_+(V_\lambda^*)$, since $[e_i,f_i]=0$ for all simple roots $\alpha_i\in R^+(\gg)$.  We obtain, using Lemma \ref{le:weight-differences} that $\delta(v_\lambda)=H\wedge v_\lambda$ and $\delta(v_{\mu})=\sum_{\alpha\in R^+(\gg)} F_\alpha\wedge v_{\alpha'}+H'\wedge v_{\mu}$, where $v_{\mu}\in V(\mu)$  and $v_{\alpha'}\in V(\alpha')$, where $\alpha'=\mu+\alpha$  and $H,H'\in \hh$. We now compute

 $$0=\delta([v_\lambda,v_{\mu}])=[H\wedge v_\lambda, \Delta(v_\mu]+[\Delta(v_\lambda),
\sum_{\alpha\in R^+(\gg)} F_\alpha\wedge v_{\alpha'}+H'\wedge v_{-\lambda}]\ .$$
$$= [H, v_{\mu}] \wedge v_\lambda+ [H', v_{\lambda}]\wedge v_{\mu} +\sum_{\alpha\in R^+(\gg)}[v_\lambda, F_\alpha]\wedge v_{\alpha'} \ .$$
Since $[F_{\alpha_i}, v_\lambda]\ne 0$ for some $\alpha_i\in \Delta$ , we obtain by comparing coefficients for the terms $[v_\lambda, F_{\alpha}]\wedge v_{\alpha}$ that $\lambda-\alpha_i=\mu+\alpha$ for some $\alpha\in R^+(\gg)\cup \{0\}$, as asserted. The lemma is proved.
\end{proof}

The proposition is proved.
 \end{proof}

 A Lie subbialgebra $\gg'$ of a Lie bialgebra $\gg$ is a Lie subalgebra of $\gg$ such that the restriction of the coalgebra structure defines a Lie algebra structure on $(\gg')^*$. We have the following fact.
\begin{proposition}
\label{pr:Levi}
\noindent(a) Let $(\gg,[\cdot,\cdot],\delta)$ be a Lie bialgebra with invariant quadratic form $\langle\cdot,\cdot\rangle$ and let $\gg'\subset \gg$ be a semisimple Lie subbialgebra of such that $\gg=\gg'\oplus (\gg')^\perp$, where $(\gg')^\perp$ denotes the orthogonal complement of $\gg'$ under  $\langle\cdot,\cdot\rangle$. If $\gg\ltimes V$ is a semidirect  Lie bialgebra  for some $\gg$-module $V$, then $\gg'\ltimes V$ also defines a semidirect  Lie bialgebra  with action given by the restriction of bracket $[\cdot,\cdot]'$ and cobracket $\delta'$ from $\gg\otimes V$ to $\gg'\otimes V$.

\noindent(b) Let $\gg$ be a semisimple complex Lie algebra and let $\ll$ be a Levi subbialgebra of $\gg$ with $\ll\cong \ll'\oplus\zz$, where $\ll'$ is semisimple and $\zz$ is central. Let $V$ be a $\gg$-module. If $\gg\ltimes V$ admits a semidirect  Lie bialgebra associated to the standard Lie bialgebra structure, then so does the restriction of the Lie algebra to $\ll'\ltimes V$.
\end{proposition}

\begin{proof}
Prove (a) first. The restriction of the bracket clearly defines a Lie bracket  $[\cdot,\cdot]'$ on $\gg'\otimes V$.
We have by definition $\delta(x)=\delta'(x)\in\gg'\wedge\gg'$ for all $x\in\gg'$ and we can write for all $v\in V$,  $\delta(v)=\sum x_i\wedge v_i+\sum x_j^\perp\wedge v_j$, where $x_i\in\gg'$ and $x_j^\perp\in(\gg')^\perp$. It is easy to see that $\delta'(v)= \sum x_i\wedge v_i$ defines a cobracket on $\gg'\ltimes V$ and that $[\cdot,\cdot]'$ and $\delta'$ satisfy \eqref{eq:cocycle}. Part (a) is proved.

Prove (b) next. The Lie subbialgebra $\ll'$ and the standard inner product clearly satisfy the conditions on $\gg'$ and  $\langle\cdot,\cdot\rangle$ in part (a). We, therefore, obtain that $\ll\ltimes V$ admits a semidirect  Lie bialgebra  and since $V$ is semisimple as an $\ll'$-module we have that $V= V'\oplus V''$ as $\ll'$-modules.  Part(b) and the proposition are proved.
\end{proof}

In the light of Proposition \ref{pr:not too big} we need the following result

\begin{proposition}
Let  $\gg$ be a simple complex Lie algebra, $\gg\notin\{sl_n,E_6\}$ for $n\ge 3$ and let $\lambda\in P^+(\gg)$. If $2\lambda-\alpha_i\in R(\gg)$ for some simple root $\alpha_i$, then $\lambda$ is one of the following:

\begin{enumerate}

\item If $\gg=sl_2$, then $\lambda=\{\omega_1, 2\omega_1\}$.
 \item  If $\gg=so(n)$ then $\lambda=\omega_1$ or if  $n=5$ and $\lambda=\omega_2$, $n=7$  and $ \lambda=\omega_3$, $n=8$ and $\lambda=\{\omega_3,\omega_4\}$or $n=10$ and $ \lambda\in\{\omega_4,\omega_5\}$.

\item If $\gg=(sp(2n))$ then $\lambda=\omega_1$ or $n=2$ and $\lambda=\omega_2$.

\item If $\gg=G_2$, then $\lambda=\omega_1$.
\end{enumerate}
\end{proposition}

\begin{proof}
The proposition was proved in detail in \cite[ch.6.1]{ZW}.
\end{proof}

On the other hand, in the case of $\gg\in\{sl_n, E_6\}$ we  obtain the following fact.
\begin{lemma}
\noindent(a)  Let $\gg=sl_n$. Then $\lambda\in\{\omega_1,2\omega_1,\omega_2,\omega_{n-2},\omega_{n-1}, 2\omega_{n-1}\}$.

\noindent (b) Let $\gg=E_6$. Then  $\lambda\in\{\omega_1,\omega_6\}$.
\end{lemma}

 \begin{proof}
 Again analogous to the argument in \cite[ch.6.1]{ZW} we consider certain Levi subalgebras.  Applying Proposition \ref{pr:Levi} and Lemma \ref{le:nottoobig} we obtain the desired result.
 \end{proof}

 We will conclude the proof with the three remaining cases, the third fundamental module of $so(7)$, the first fundamental module of $G_2$ and, most interestingly, the first fundamental module of $sp(2n)$.

 \begin{lemma}
 Let $\gg=so(7)$ and $V=V_{\omega_3}$. Then $\gg\ltimes V$ does not admit a semidirect Lie bialgebra structure corresponding to the standard Lie bialgebra structure on $so(7)$.
 \end{lemma}

 \begin{proof}
 Suppose there was. Note that $\omega_3=\frac{1}{2}(\alpha_1+2\alpha_2+3\alpha_3)$ and $\alpha_{max}=\alpha_1+2\alpha_2+2\alpha_3$, the highest root.  Let $v_{\omega_3}\in V(\omega_3)$ and $v_{-\omega_3}\in V(-\omega_3)$. Clearly, $F_\alpha(v_{\omega_3})\notin V(-\omega_3)$ and $E_\alpha(v_{-\omega_3})\notin V(\omega_3)$ for all $\alpha\in R(\gg)$. As in the proof of Lemma \ref{le:nottoobig}  we may assume that $\delta(v_{\omega_3})=h\wedge v_{\omega_3}$ and $\delta(v_{-\omega_3})=h'\wedge v_{-\omega_3} +\sum_{\alpha\in R^+(\gg)} F_\alpha\wedge v_{\alpha}$ where $v_\alpha\in V(-\omega_3+\alpha)$ and $h,h'\in \hh$.
 Therefore $0=\delta([v_{\omega_3}, v_{-\omega_3}])$  implies
$$ [h\wedge v_{\omega_3}, \Delta(v_{-\omega_3})]+[\Delta(v_{\omega_3}),h'\wedge v_{-\omega_3}]= (h,-\omega_3) v_{-\omega_3}\wedge  v_{\omega_3}+ (h',\omega_3) v_{-\omega_3}\wedge v_{\omega_3}\ .$$

 We thus obtain that $(h,\omega_3)=(h',\omega_3)$, even though  Lemma \ref{le:weight-differences} implies that $h-h'=2\omega_3$, which leads to a contradiction.
 The lemma is proved.
 \end{proof}

The case of the first fundamental module of $G_2$ can be proved by a similar argument.

Now we shall consider the case of the first fundamental module of $sp(2n)$.

  \begin{proposition}
  \label{pr: sp 2n}
  Let $\gg=sp(2n)$  with Belavin-Drinfeld Lie bialgebra structure and $V=V_{\omega_1}$. Then $\gg\ltimes V$ does not admit a  semidirect  Lie bialgebra structure.
  \end{proposition}
\begin{proof}
 Indeed consider the Lie algebra $\gg'=sp(2n+2)$ and the parabolic subalgebra corresponding to removing the first node of the Dynkin diagram. The semisimple part of the Levi is isomorphic to $\gg=sp(2n)$ and the radical is isomorphic as a $\gg$-module to $\nn=V_{\omega_1}\oplus V_0$, $V_0$ corresponding to the maximal root space of $sp(2n+2)$ We extend the Lie bialgebra structure from $sp(2n)$  to $sp(2n)\ltimes \nn$, as in the proof of Proposition \ref{pr:abelian rads-lie bi}.  Consider the action of the double $D(sp(2n))$ on $\nn$. We obtain that (up to a permutation of factors) $V_{\omega_1}\oplus V_0\cong V_{\omega_1}\otimes V_0\oplus V_0\otimes V_0$  as a $D(sp(2n))$-module. The action of $D(sp(2n))$ determines
 $$[\delta(v),v'\otimes 1+1\otimes v']+[v\otimes 1+1\otimes v,\delta(v')]\ ,$$
 for all $v,v'\in V_{\omega_1}\otimes V_0$.  The radical $\nn$, however, is not Abelian and  for $v\in V_{\omega_1}(\omega_1)$ and $v'\in  V_{\omega_1}(-\omega_1)$ one has $[v,v']\in V_0\ne 0$ and $\delta([v,v'])\ne 0$. Hence, if $V_{\omega_1}\otimes V_0$ is assumed to an Abelian subalgebra of $D(sp(2n))\ltimes (V_{\omega_1}\otimes V_0)$, then the cocycle identity \eqref{eq:cocycle} cannot be satisfied for all $v,v'\in V_{\omega_1}\otimes V_0$. Therefore, there is no  semidirect Lie bialgebra structure on $sp(2n)\ltimes V_{\omega_1}$.
\end{proof}

 It remains to show that there exist associated Lie bialgebra structures for the modules listed in Theorem \ref{th:main1}.  This was accomplished in the proof of Theorem \ref{th:main} by explicitly describing them using Proposition \ref{pr:abelian rads-lie bi}. We can now conclude the proof as in the proof of Theorem \ref{th:main}.
 Theorem \ref{th:main1} is proved.
 \end{proof}

\section{Co-Poisson  Module Algebras and  Locally $ad$-finite Subalgebras of $U_q(g)$}
\label{se:co-dec-locadfin}
This section is devoted to the construction of another interesting class of  co-Poisson Hopf algebras which arise of the classical limits of $U_q(\gg)\ltimes  A$ where $U_q(\gg)$ denote the standard quantized universal enveloping algebra of a semisimple Lie algebra $\gg$ and $A\subset U_q(\gg)$ is a filtered $U_q(\gg)$-module subalgebra such that each filtered component is a finite-dimensional $U_q(\gg)$-module. We, then, show in Section \ref{se:qu-lin-qsas} that we obtain through our construction a large family of interesting solutions to Problem \ref{pr:equiv-defs}, and provide explicit calculations for an example in Section \ref{se:the natural representation}.
 \subsection{The finite part of $U_q(\gg)$ }

 We construct in this section finitely graded $U_q(\gg)$-module algebras using the description of the locally finite part of the quantized enveloping algebra $U_q(\gg)$ (see Section \ref{se:appendix2}) by Joseph and Letzter in \cite{JL1} and \cite{JL}. Moreover, we will show how they can be interpreted as quantizations of  co-Poisson Hopf algebras.

   Let $\gg$ be a complex simple Lie algebra and $V$ a $k$-dimensional $U_q(\gg)$-module. Recall that if $V$ is  a $U_q(\gg)$-module with $\rho:U_q(\gg)\to End_k(V_\lambda)$, then $U_q(\gg)$ acts on $E_\lambda=End_k(V_\lambda)$  by conjugation $x(\phi)=\rho(x_{(1)})\phi \rho(S(x_{(2)})$.
  Recall the following result of Joseph and Letzter describing the integrable -- locally finite under the adjoint action-- subspace
 $$\mathfrak{F}=\{x\in U_q(\gg): ad (U_q(\gg))(x) \text{ is finite dimensional}\}\subset U_q(\gg)\ .$$

    \begin{theorem}(Joseph and Letzter \cite{JL})
 Let  $\lambda\in P^+(\gg)$ be a dominant integral weight of $\gg$. Then there exists an injective $U_q(\gg)$-module homomorphism $i: E_\lambda\to  U_q(\gg)$. Moreover
 $$\mathfrak{F}\cong \sum_{\lambda\in P^+(\gg)} E_\lambda\ .$$
 \end{theorem}

 Of particular interest for our discussion will be a result by Lyubashenko and Sudbery \cite{Lyu-Sud}. We have $ E_\lambda\cong \hat L_\lambda\oplus k\cdot  Id_{V_\lambda}$, where $L_\lambda$ is a $(k^2-1)$-dimensional $U_q(\gg)$-module. Denote by $C_\lambda'=i(Id_{V_\lambda})$ and $L_\lambda=i( \hat L_\lambda)$.

 \begin{theorem}(Lyubashenko and Sudbery \cite{Lyu-Sud})
 \label{th:ad-finite-alg}
 Let  $\lambda\in P^+(\gg)$ be a dominant integral weight and $L_\lambda\subset U_q(\gg)$ the $U_q(\gg)$-module defined above.  Then there exist a linear map $\sigma:L_\lambda\otimes L_\lambda\to L_\lambda\otimes L_\lambda$ and a central element $C_\lambda=c\cdot C_\lambda' \in U_q(\gg)$, $c\in k$ such that for all $x,y\in L_\lambda$

 $$xy-yx-\mu\circ\sigma(x\otimes y)=  ad(x)(y)C_\lambda\ .$$
 \end{theorem}
In the following we will sketch the proof of Lyubashenko and Sudbery.
 They first show that $L_\lambda\subset U_q(\gg)$ is  left coideal in $U_q(\gg)$, hence we conclude that the subalgebra $A_\lambda\subset U_q(\gg)$  generated by $L_\lambda$ is a left coideal algebra.  One can show in particular that there exists a central element $C_\lambda\in U_q(\gg)$ such that
  \begin{equation}
 \label{eq:co-product-Lyu}
 \Delta(x)= x\otimes C_\lambda+\sum u'\otimes x'\ ,
 \end{equation}
where $u\in U_q(\gg)$ and $x\in L'_\lambda$. Using the Hopf algebra identity
 $$xy= ad(x_{(1)})( y )\cdot x_{(2)}= x_{(1)} yS(x_{(2)}) x_{(3)}$$
 we can now define $\sigma$ using \eqref{eq:co-product-Lyu}

 \begin{equation}
 \label{eq:def.of sigma}
 \sigma(x\otimes y)= x_{(1)} yS(x_{(2)}) x_{(3)}-ad(x)(y)\otimes C_\lambda\ .
 \end{equation}

 The assertion of Theorem \ref{th:ad-finite-alg} follows.

 We have the following fact.

 \begin{theorem}
 Let $A_\lambda\subset U_q(\gg)$ be as above.
 There exists a $U(\gg)$ module bialgebra $\overline A_\lambda$ such that $ \overline A_\lambda\equiv A_\lambda (mod\ (q-1))$ with a co-Poisson co-decoration
 $\delta:\overline A_\lambda\to U(\gg)\wedge \overline A_\lambda$ satisfying
 $\delta(\overline a)\equiv\frac{ \Delta(a)-\Delta^{op}(a)}{q-1} (mod\ (q-1))$ if $\overline a\equiv a (mod\ (q-1))$.
 \end{theorem}
 \begin{proof}
  The theorem follows directly from  Proposition \ref{pr:c-l-algebra} and Proposition \ref{pr:c-l--coPoisson} by employing the fact that $A_\lambda$ is a left coideal algebra. The theorem is proved.
 \end{proof}

 \begin{remark}
 It is now possible to compute the co-Poisson structures explicitly. In section \ref{se:the natural representation} we show an example of such a computation  in the $sl_2$ -case.\end{remark}

\begin{question}
Describe the co-Poisson module algebra structures explicitly. In particular classify all those structures which are linear; i.e. where $\delta: L_\lambda\to U(\gg)_1\wedge L_\lambda$, where $U(\gg)_1$ denotes the first filtered component of $U(\gg)$, thus analogous to Lie bialgebra structures.
\end{question}
Indeed, it is relatively easy to observe that if $V$ is a $\ell$-dimensional  simple $U(sl_2)$-module, then the co-Poisson  module algebra  structures on $L=End(V)$ are non-linear if $\ell\ge 3$.

\subsection{Filtrations and Quantized Symmetric Algebras}
\label{se:qu-lin-qsas}
We will show in this section how we can employ the co-Poisson  module algebra structures introduced above to construct quantizations of the symmetric algebras of certain $U(\gg)$-modules where $\gg$ is a semisimple Lie algebra, analogous to the Donin's argument \cite{Don}.
Note that  the central element $C_\lambda$ defined in Theorem \ref{th:ad-finite-alg} is invertible and that $C_\lambda^{-1}$ is central as well for each $\lambda\in P(\gg)$.  Consider the $U_q(\gg)$ module homomorphism $\phi=C_\lambda^{-1}\cdot i:L_\lambda'\to U_q(\gg)$ which extends naturally to the tensor algebra $\phi:T(L_\lambda')\to U_q(\gg)$.  Recall that an algebra $U$ is called filtered, if $U=\bigcup_{i=0}^\infty U_i$ with $U_i\subset U_{i+1}$  and $U_i\cdot U_{j}\subset U_{i+j}$. We have the following fact.

\begin{lemma}
\label{le:qu-linear in ad}
The $U_q(\gg)$-module algebra $\phi(T(L_\lambda'))$ is a filtered algebra, with the filtration defined by powers of $C_\lambda^{-1}$.
\end{lemma}

\begin{proof}
Indeed, we have in $\phi(T(L_\lambda'))$  the following  relations: For all $x,y\in L_\lambda$:
$$C_\lambda^{-2} x y- C_\lambda^{-2} \mu(\sigma(x\otimes y)=C_\lambda C_\lambda^{-2} ad(x)(y)=C_\lambda^{-1}ad(x)(y)\ .$$

These relations then naturally induce a quadratic linear filtration on  $\phi(T(L_\lambda'))$.
\end{proof}

Recall, additionally, that the associated graded algebra $gr(U)$ of a filtered algebra $U$ is defined as $gr(U)=\bigoplus_{i=0}^{\infty} U_{i}/U_{i-1}$, where we set $U_{-1}=\{0\}$.
We now obtain from Lemma \ref{le:qu-linear in ad} the following main result of this section.

\begin{theorem}
\label{th:qsa-general}
   The associated graded algebra $S_q(L'_\lambda)$ of $\phi(T(L_\lambda'))$ is a quantization of the symmetric algebra $S(\overline{L'_\lambda})$, where $\overline{L'_\lambda}$ denotes the $U(\gg)$-module which is the classical limit of $L'_\lambda$.
\end{theorem}

\begin{proof}
  We need the following well-known fact.
 \begin{lemma}
 \label{le: grmodal}
 Let $A$ be a Hopf algebra, and $U$ be a filtered $A$-module algebra. Then $\phi: U\to gr(U)$  is an isomorphism of $A$-modules.
 \end{lemma}
Hence, $S_q(L_\lambda')$ is a $U_q(\gg)$-module algebra, and it remains to consider the classical limit.  First we have to establish its existence in the terms of $\UU=(U_q(\gg), U_A(\gg))$-module algebras (for notation and the construction of the classical limit see Section \ref{se:appendix}). We need the following fact .

\begin{lemma}
Let $L'_{A,\lambda}$ be an $A$-lattice in  $L'_\lambda$. Then  the restriction of $\sigma$ to $(L'_{A,\lambda})^{\otimes 2}$ defines an $U_A(\gg)$ module algebra homomorphism from $\sigma:(L'_{A,\lambda})^{\otimes 2}\to(L'_{A,\lambda})^{\otimes 2}$.
\end{lemma}

\begin{proof}
Note that $U_A(\gg)$ is a sub Hopf algebra of $U_q(\gg)$ by Lemma \ref{le:Hopf pair}, in particular it is closed under multiplication and comultiplication, and $U_A(\gg)$ acts adjointly on itself. Then \eqref{eq:def.of sigma} implies that $\sigma(L'_A(\lambda)^{\otimes 2})\subset L'_A(\lambda)^{\otimes 2}$. The lemma is proved.
\end{proof}

We now obtain that the quotient of $(T(L'_\lambda), T(L'_{A,\lambda}))$ by the ideal generated by $(1-\sigma)$ is a $\UU$-module algebra  $(S_q(L'_\lambda),S(L'_{A,\lambda}))$.
This allows us to consider the classical limit. It now follows from the PBW theorem for $U_q(\gg)$ that  the classical limit of $\overline {S_A(L'_\lambda)}\cong S(\overline L'_\lambda)$ as $U(\gg)$-modules. Theorem \ref{th:qsa-general} is proved.
\end{proof}

In the case of $\gg=sl_n$, where $L_\lambda'\cong V_{ad}$, where $V_{ad}$ is the adjoint module  we have the following corollary which would be originally due to Donin.

\begin{corollary}
\label{cor: sl_n-ad}(Donin  \cite{Don})
The adjoint representation of the Lie algebra $sl_n$ has a quantum symmetric algebra $S_q(V_{ad})$.
 \end{corollary}

Note that $L_\lambda'$ is not simple for all other choices of $\gg$ and $\lambda\in P^+(\gg)$. Hence this construction will not directly yield quantum symmetric algebras for simple $U_q(\gg)$-modules. However,  if $Hom(V_\mu, L_\lambda')\ne \{0\}$, then one may wish to consider  the subalgebras of $S_q(L'_\lambda)$ generated by copies of $V_\mu$ as the quantum symmetric algebras $S_q^\lambda(V_\mu)$. However, this leads to the following obvious question.

\begin{question}
Suppose that $Hom(V_\mu, L_\lambda')\ne \{0\}$ and $Hom(V_\mu, L_{\lambda'}')\ne \{0\}$ for some $\lambda, \lambda',\mu\in P^+(\gg)$. Are $S_q^\lambda(V_\mu)$ and $S_q^{\lambda'}(V_\mu)$ isomorphic as algebras?
\end{question}

\subsection{The natural representation of $sl_2$}
\label{se:the natural representation}
In this section we will review the constructions discussed above for the example $L_{\frac{1}{2}}=End(V_{\frac{1}{2}})$, where $V_{\frac{1}{2}}$ denoted the two-dimensional simple  $U_q(sl_2)$-module.  Lyubashenko and Sudbery \cite[Equations 3.1 and 3.2]{Lyu-Sud} show that $L_{\frac{1}{2}}\subset U_q(sl_2)$ has basis

$$ X_+=K^{-1} E\ ,\quad X_-=K^{-1}F\ ,\quad X_0=\frac{q EF-q^{-1} FE}{q-q^{-1}}\ ,$$
$$ C=K^{-1}+\frac{q-q^{-1}}{q+q^{-1}}(qEF-q^{-1} FE)\ ,$$
where $E,F, K^{\pm1}$ denote the standard generators of $U_q(sl_2)$ and $X_+,X_-$ and $X_0$ span $L'_{\frac{1}{2}}$ and $C=C_{\frac{1}{2}}$ denotes the central element.

This allows us (using the definitions of Section \ref{se:appendix2}) to compute the co-Poisson  structure on the classical limit $\overline A_{\frac{1}{2}}$ of $A_{\frac{1}{2}}$.

\begin{proposition}
The co-decoration on $\overline A_{\frac{1}{2}}$ is the map $\delta_{\frac{1}{2}}: \overline A_{\frac{1}{2}}\to U(\gg)\wedge \overline A_{\frac{1}{2}}$ given by
$$\delta_{\frac{1}{2}}(X_+)= H\wedge X_+\ ,\quad \delta_{\frac{1}{2}}(X_-)=H\wedge X_-\ ,\quad \delta_{\frac{1}{2}}(X_0)=H\wedge X_0+E\wedge X_-+F\wedge X_+\ .$$
\end{proposition}

\begin{proof}
The proposition is proved by straightforward computation.
\end{proof}

Additionally of interest are also the structures on the filtered algebra $gr(A_{\frac{1}{2}})$ and its classical limit.

Using the fact that $C=K^{-2}+(q-q^{-1})X_0$ we now compute:
 $$\Delta(X_\pm)=X_+\otimes K^{-2}+1\otimes X_\pm=X_\pm\otimes C- (q-q^{-1}) X_\pm\otimes X_0+1\otimes X_\pm\ .$$

We now obtain
$$\sigma(X_+\otimes X_-)=(q-q^{-1})ad(X_+)(X_-)\otimes X_0+X_-\otimes X_+=(q-q^{-1})(q+q^{-1}) X_0\otimes X_0+X_-\otimes X_+\ .$$

Similarly,
$$\sigma(X_+\otimes X_0)=(q-q^{-1}) ad(X_+)(X_0)\otimes X_0+ X_0\otimes X_+=-(q-q^{-1})q^{-1} X_+\otimes X_0+X_0\otimes X_+\ .$$

Analogously we compute that
$$\sigma(X_-\otimes X_0)=q(q-q^{-1})X_-\otimes X_0+X_0\otimes X_-\ .$$

We can now immediately derive the corresponding Poisson bracket.

\begin{proposition}
The Poisson structure on $S(\overline L'_{\frac{1}{2}})$, i.e. the symmetric algebra on the adjoint representation of $sl_n$ is given (up to a nonzero scalar) by
$$\{X_+,X_0\}=-X_+X_0\ ,\quad \{X_+,X_-\}=2X_0^2\ ,\quad \{X_-,X_0\}=X_-X_0\ .$$
\end{proposition}

\begin{proof}
Straightforward Computation.
\end{proof}

\begin{remark}
Lyubashenko and Sudbery show in  \cite[Theorem 3.1]{Lyu-Sud} that the algebra $A_{\frac{1}{2}}\subset U_q(sl_2)$ is isomorphic as a $U_q(sl_2)$-module algebra to the $ad$-finite part of $U_q(sl_2)$.
\end{remark}

   \section{Quantized Symmetric Algebras and Co-boundary categories}

In this section, we will define quantized symmetric algebras associated to semidirect Lie bialgebras $\gg\ltimes V$, and, if $\gg$ satisfies certain conditions, describe them as symmetric algebras in a coboundary category of $U_h(\gg)$-modules. For notation and more on the quantization of Lie bialgebras see Section \ref{se:quant.of LBA}.
   \subsection{Quantized Symmetric Algebras}

  This section is devoted to proving Theorem \ref{th:qia} which establishes the existence of quantum symmetric algebras as semidirect factorizations of $U_h(\gg\ltimes V)$. Denote by a sub-Manin triple $(\gg,\gg_+,\gg_-)$ of a Manin triple $(\gg',\gg_+',\gg_-')$ a Manin triple such that $\gg_+\subset \gg_+'$, $\gg_-\subset \gg_-'$ and such that  the invariant bilinear form $\langle \cdot,\cdot\rangle$ on $\gg$ is the restriction of the bilinear form  $\langle \cdot,\cdot\rangle'$on $\gg'$. Now we are able to state the theorem.
  \begin{theorem}
  \label{th:qia}
  Let $(\gg',\gg_+',\gg_-')$ be a Manin triple of Lie bialgebras and let $(\gg, \gg_+,\gg_-)$ be a sub-Manin triple such that $\gg_+$ is semisimple. Let $V$ be a $\gg_+$-stable subspace of $\gg'$. Then there exists an associative $U_h(\gg_+)$-module algebra $U_h(V)$ such that $U_h(V)/hU_h(V)=U(V)$, the subalgebra of $U(\gg')$ generated by $V$.
  \end{theorem}

  \begin{proof}
  Recall that a Hopf algebra $H$ acts on itself via the adjoint action $ad(h)(x)=h_{(1)}xS(h_{(2)})$ for all $h,x\in H$, where we use Sweedler notation $\Delta(h)=h_{(1)}\otimes h_{(2)}$.  Indeed the adjoint action gives $H$ the structure of an $H$-module algebra. Hence $U_h(\gg_+)\subset U_h(\gg)$ acts adjointly on  the preferred quantization $U_h(\gg)=U(\gg)[[h]]$, as algebras. Now, $U_h(\gg_+)$ is equivalent to $A=((U(\gg_+)[[h]], \Delta,\varepsilon,\Phi_{KZ}, \tilde S)$ with equivalence $(\Theta,J)$. The quasi-bialgebra $A$ acts naturally on $U_h(\gg)[[h]]$ and $V[[h]]$ is, therefore,  an object of $Rep(A)$. We have an equivalence of  braided tensor categories $\mathcal{F}:Rep(A)\to Rep( U_h(\gg_+))$, the category of  $U_h(\gg_+)$-modules, and of the corresponding categories of  quasi-associative module algebras by Proposition \ref{pr:equiv-qamas}. Thus $U(\gg)[[h]]$ is given the structure of an $A$-module algebra. Consider the $A$-module algebra $U(V[[h]])\subset U(\gg)[[h]]$ generated $V[[h]]$, which can be given a  quasi-associative  $U_h(\gg_+)$-module algebra structure  by Proposition \ref{pr:equiv-qamas}.  It only remains to observe that clearly $U(V[[h]])/hU(V[[h]])=U(V)$. The twist $(\Theta,J)$ defines an associative $U_h(\gg_+)$-module algebra structure on a twist  $U_h(V)$ of $U(V[[h]])$, because $U_h(\gg_+)$ is associative. Note that $J\in (U(\gg_+)[[h]])^{\otimes 2}$ is invertible, hence $J\equiv 1\otimes 1(mod\ h)$. This implies that $U(V[[h]])/hU(V[[h]])=U(V)$ as a $U_h(\gg_+)/h U_h(\gg_+)=U(\gg_+)$-module algebra. The theorem is proved.
    \end{proof}
\begin{definition}
If in Theorem \ref{th:qia} we have that $V$ is Abelian, hence $U(V)=S(V)$, then we call $U_h(V)=S_h(V)$ a quantum symmetric algebra over $U_h(\gg_+)$.
\end{definition}

\begin{remark}
Theorem \ref{th:qia} can be straightforwardly generalized to the case where $\gg_+$ satisfies $H^2(\gg_+,\gg_+)=0$.
\end{remark}

\subsection{Symmetric algebras and co-boundary categories}
We will explain in this section how quantizations of  semidirect Lie bialgebras can be naturally interpreted as symmetric algebras in co-boundary categories. This will also show the relations between the quantizations of semidirect Lie bialgebras and the braided symmetric algebras introduced by A. Berenstein and the author in \cite{BZ}. We will first recall the definition of a co-boundary Hopf algebra.

\begin{definition}
A co-boundary Hopf algebra is a pair $(H,\R)$ of a Hopf algebra $H$ and an invertible element $\R\in H\otimes H$ satisfying the following relations:
\begin{equation}
\label{eq:co-bound 1}
\Delta^{op}=\R \Delta\R^{-1}\ ,\quad \R \R_{op}=1\otimes 1\ ,\end{equation}
\begin{equation}
\label{eq:co-bound 2}
\R^{op}(\Delta\otimes Id)\R=\R^{23} (Id\otimes \Delta)\R\ ,\end{equation}
\begin{equation}
\label{eq:co-bound 3}
(\varepsilon\otimes Id)\R=(id\otimes \varepsilon)\R=1\ .\end{equation}
\end{definition}

Additionally recall the definition of a co-boundary category.

\begin{definition} A coboundary category is an Abelian monoidal category with natural isomorphisms $\sigma_{A,B}:A\otimes B\to B\otimes A$ for all objects $A$ and $B$ such that
$$\sigma_{B,A}\circ\sigma_{A,B}=Id_{A\otimes B}$$
 \begin{equation}
\label{eq:comm square}
\begin{CD}
A\otimes B \otimes C@>\sigma_{12,3}>>C\otimes A \otimes B\\
@V\sigma_{1,23}VV@VV\sigma_{23}V\\
B\otimes C \otimes A@>\sigma_{12}>>C\otimes B \otimes A
\end{CD}
\end{equation}
where we abbreviated $\sigma_{12,3}=\sigma_{A\otimes B,C}$ etc.
 \end{definition}

The following fact is well known.

\begin{proposition}
Let $(H,\R)$ be a coboundary Hopf algebra. The coboundary element $\R$ defines a coboundary structure on the category of $H$-modules via
$\sigma_{U,V}=\tau\circ \R$, where $\tau$ denotes the permutation of factors.
\end{proposition}

Moreover, Enriquez and Halbout recently proved the following result in \cite{EH}.

\begin{theorem}\cite{EH}
Let $\gg$ be a co-boundary Lie bialgebra with co-boundary element $r^-\in\Lambda^2 \gg$. Then there exists a quantized universal enveloping algebra $U_h(\gg)$ quantizing $\gg$ with an element $\R$ satisfying \eqref{eq:co-bound 1}--\eqref{eq:co-bound 3} such that
$$\frac{R-R^{op}}{h}\equiv 2 r \ (mod\ h)\ .$$
\end{theorem}

 Coboundary categories, sometimes under the name "cactus categories",  have recently attracted interest as the categories in which Henriques  and Kamnitzer define crystal commutors  in  \cite{HK}. Kamnitzer and Tingley then studied the relationship with the coboundary categories associated to coboundary Hopf algebras in \cite{KT1}. Co-boundary categories are particularly interesting because they allow for natural notions of symmetric and exterior algebras and powers.

\begin{definition}Let $V$ be an  object in a linear coboundary category $\mathcal{C}$ over a field $k$ with $char(k)\ne 2$).

\noindent(a)  Define the symmetric square of $V$ in $\mathcal{C}$ to be  $S^2_\sigma V= (\sigma+Id) (V\otimes V)$.

\noindent(b) Similarly, define the exterior square in $\mathcal{C}$ as $\Lambda^2V=(\sigma-Id) (V\otimes V)$.
\end{definition}

We have the following fact.
\begin{lemma} Let $V$ be an  object in a linear coboundary category $\mathcal{C}$ ($char(k)\ne 2$). $$V\otimes V\cong_{\mathcal{C}} S^2_\sigma V\oplus \Lambda^2V\ .$$
\end{lemma}

\begin{proof} The endomorphism $\sigma_{V,V}$ is an involution, hence it is semisimple and its eigenvalues are $\pm 1$. The lemma is proved.
\end{proof}

We can now define symmetric and exterior algebras, as well as higher symmetric and exterior powers for objects in linear coboundary categories.

\begin{definition}
\label{def: cob-sym-ext-alg-pow}
Let $V$ be an object in a linear coboundary category  $\mathcal{C}$($char(k)\ne 2$).

\noindent(a) Define the symmetric algebra of $V$ in $\mathcal{C}$ to be
$$S_\sigma(V)=T(V)/\langle \Lambda^2_\sigma V \rangle\ ,$$
where $\langle \Lambda^2_\sigma V \rangle$ denotes the ideal generated by $\Lambda^2_\sigma V$.
Similarly, define the exterior algebra of $V$ in $\mathcal{C}$ to be
$$S_\sigma(V)=T(V)/\langle S^2_\sigma V \rangle\ .$$

\noindent(b)  Define the $n$th symmetric power $S_\sigma^nV \subset V^{\otimes n}$ and the  $n$-th exterior power  $\Lambda_\sigma^nV\subset V^{\otimes n}$ ($n\ge 2$) by:

$$S_\sigma^nV=\bigcap_{1\le i\le n-1}  \left(Ker~ (\sigma_{i,i+1}-id)\right)=\bigcap_{1\le i\le n-1} \left(Im~(\sigma_{i,i+1}+id)\right)\ ,$$
$$\Lambda_\sigma^nV=\bigcap_{1\le i\le n-1} \left(Ker~(\sigma_{i,i+1}+id)\right)=\bigcap_{1\le i\le n-1} \left(Im~ (\sigma_{i,i+1}-id)\right) ,$$
where we abbreviated $\sigma_{i,i+1}=Id^{\otimes (i-1)}\sigma_{V,V}^{i,i+1}\otimes Id^{\otimes (n-1-i}$.
Define the symmetric and exterior powers for $n=0,1$ by:
$$S^0_\sigma V=k\ ,\quad S^1_\sigma V=V\ , \quad \Lambda^0_\sigma V=k\ ,\quad \Lambda^1_\sigma V=V\ .$$
\end{definition}

 \subsection{Quantized symmetric algebras and co-boundary categories}
\label{se:qsa-coboundary}
The following theorem is the main result of this section.

\begin{theorem}
\label{th:qsa-cobound}
Let $\gg$ be a semisimple coboundary Lie bialgebra with coboundary element $r^{-}\in\Lambda^2 \gg$ and $\gg\ltimes V$ a  semidirect Lie bialgebra. Then the algebra $S_h(V)$ is the symmetric algebra of the $U_h(\gg)$-module $V[[h]]$ in the coboundary category of modules over the quantization of $(\gg, r^-)$.
\end{theorem}

\begin{proof}
 Recall that the double of a Lie bialgebra is naturally a quasitriangular Lie bialgebra.  Next, recall that if $\gg$ is a quasitriangular Lie bialgebra with classical $r$-matrix $r$, then $r^{-}=\frac{1}{2}(r-r^{op})$ endows $\gg$ with a coboundary structure, since
$$ \delta(x)=[r,x\otimes 1+1\otimes x]=[r^{-}, x\otimes 1+1\otimes x]\ .$$
We need  the following proposition.
\begin{proposition}
\label{pr:co-boundary-double}
Let $\gg$ be a semisimple Lie bialgebra and $\gg\ltimes V$ a  semidirect Lie bialgebra. Then the   algebra $S_h(V)$ is a symmetric algebra in the coboundary category $\C_{D}$ defined by $U_h(D(\gg))$.
\end{proposition}

\begin{proof} Recall from Theorem \ref{th:twist-equivalence} that $((U(D(\gg))[[h]], \Delta,\varepsilon,\Phi_{KZ}, \tilde S, R)$ and $U_h(D(\gg)$ are twist equivalent with respect to a twist $J\in U_h(D(\gg))^{\otimes 2}$. Denote by $\tilde r$ the canonical element of $D(\gg)$. The twist $J$ has the following property.

\begin{lemma}
The element $J^{op}J^{-1}$ satisfies \eqref{eq:co-bound 1}--\eqref{eq:co-bound 3}. Moreover, it is a quantization of $2\tilde r^{-}=\tilde r-\tilde r^{op}$.
\end{lemma}

\begin{proof}
Note that $J^{op}$ defines the twist equivalence of $(U(D(\gg))[[h]], \Delta,\varepsilon,\Phi_{KZ}, \tilde S)$ to $U_h(D(\gg))^{cop}$, where $U_h(D(\gg))^{cop}$ denotes coopposite quantized enveloping algebra of $U_h(D(\gg))$. Therefore, $J^{op} J^{-1}$ defines a twist equivalence between $U_h(D(\gg))$ and $U_h(D(\gg))^{cop}$. This implies, by the recent results of Enriquez and Halbout \cite{EH}   that $J^{op} J^{-1}$ is a quantization of the Lie bialgebra twist $2\tilde r^{-}$, see e.g. the introduction of \cite{EH}. It then follows from the definition of a twist that $J^{op}J^{-1}$ satisfies \eqref{eq:co-bound 1}--\eqref{eq:co-bound 3}. The lemma is proved. \end{proof}

Recall from the proof of Theorem  \ref{th:qia} that the relation  $xy-yx=0$ in the  $((U(D(\gg))[[h]], \Delta,\varepsilon,\Phi_{KZ}, \tilde S)$-module algebra $S(V[[h]])$ yields the relation $\mu(J(x\otimes y)-J(y\otimes x))=0$ which then implies that
$$0= \mu\left(J J^{-1} (x\otimes y)-J \circ \tau\circ J^{-1}(x\otimes y)\right) =\mu\left(x\otimes y-\tau\circ  J^{op} J^{-1} (x\otimes y)  \right)\ .$$
Since  $\sigma_{\C_{D}}=\tau\circ J^{op} J^{-1}$, we obtain as the new relation $\mu(x\otimes y-\sigma_{\C_{D}}(x\otimes y))=0$, hence the algebra $S_h(V)$ is a symmetric algebra in the coboundary category $\C_D$. The proposition is proved.
\end{proof}

From the proof of Proposition \ref{pr:co-boundary-double} we derive that the twist $J^{op}J^{-1}$ also twists $U_h(\gg)$ to $U_h(\gg)^{cop}$. Therefore it defines the same coboundary structure on the category of $U_h(\gg)$-modules as the quantization of the co-boundary element obtained by  Enriquez and Halbout \cite{EH}. Theorem \ref{th:qsa-cobound} is proved.
\end{proof}

Symmetric and exterior algebras and powers in the coboundary category associated to the standard quantized universal enveloping algebras $U_q(\gg)$ of a reductive complex Lie algebra were introduced by A. Berenstein and the author in \cite{BZ} under the name {\it braided symmetric and exterior algebras}, resp. powers and further investigated by the author in \cite[ch. 4]{ZW}. It was shown that braided symmetric algebras are in some sense more generic than classical symmetric algebras, in particular there are only a relatively small number of examples where the braided symmetric algebras are flat deformations of the classical symmetric algebras. Such modules were called {\it flat}. Theorem \ref{th:qsa-cobound} has the following immediate consequence, which agrees with Theorem 1.2 of \cite{ZW}.

\begin{corollary}
\label{cor:flatness}
Let $\gg$ be a reductive Lie bialgebra with the standard Lie bialgebra structure and let $\gg\ltimes V$ be a semidirect  Lie bialgebra, and let $V^q$ be a $U_q(\gg)$-module such that its classical limit is $V$. Then, $V^q$ is flat.
\end{corollary}
\begin{remark}
Note, that the converse of Corollary \ref{cor:flatness} does not hold: the first fundamental module of $U_q(sp(2n))$ is flat, Proposition \ref{pr: sp 2n}, however, yields that  its classical limit does not admit a Belavin-Drinfeld Lie bialgebra structure by Proposition \ref{pr: sp 2n}.
\end{remark}
\section{Examples of Quantized Symmetric Algebras}
\label{se:examples}
 In this section we will describe well known quantum algebras which can be obtained as quantum symmetric algebras. First we have the following.

 \begin{theorem}
 \label{th:descr. of qsa}
 Let $\gg$ be a simple complex Lie bialgebra with standard bialgebra structure and $V$ a simple $\gg$-module such that $\gg\ltimes V$ is a semidirect  Lie bialgebra. The following quantized function algebras are obtained as quantizations of $\gg\ltimes V$:

 \begin{itemize}
\item If $\gg=sl_m\times sl_n $ and $V=V_{\omega_1}$, then one obtains  the algebra of quantum $m\times n$-matrices.

\item If $\gg=sl_n$ and $V=V_{2\omega_1}$,  then  one obtains the  algebra of quantum symmetric matrices introduced by Nouri in \cite[Theorem 4.3 and Proposition 4.4]{Nou} and by  Kamita \cite{Kam}.

\item  If $\gg=sl_n$ and $V=V_{\omega_2}$, then  one obtains the  algebra of quantum antisymmetric matrices introduced by Strickland in \cite[Section 1]{Str}.

\item If $\gg=so(n)$ and $V=V_{\omega_1}$, then  one obtains odd- and even-dimensional quantum Euclidean space introduced by Reshetikhin et al. in \cite{RTF} (see also \cite{M}).

\end{itemize}
 \end{theorem}

 \begin{proof}
We showed  in the proof of Theorem \ref{th:main1} that the modules $V$ listed above correspond to Abelian radicals. We explicitly described  in \cite{ZW} the corresponding quantum symmetric algebras $S_q(V^q)$ as braided symmetric algebras in the sense of \cite{BZ}. Taking  the classical limit one obtains a Poisson structure on $S(V)$ via

$$\{u,v\}=\lim\limits_{q\to 1}\frac{u^qv^q-v^qu^q}{q-1} \ ,$$

for all $u,v,u^q,v^q$ such that $u^q\to u$ and $v^q\to v$ as $q\to 1$.  Since each deformation of a commutative algebra is determined by the Poisson structure it defines, it is sufficient to compare the Poisson structures obtained from the quantum symmetric algebras with the Poisson structures obtained from the quantized function algebras listed in Theorem  \ref{th:descr. of qsa}.  This was done in the proof of  \cite[Corollary 4.26]{ZW} using the results of Goodearl and Yakimov in \cite[ch. 5]{GY}.
This completes the proof of the theorem.
 \end{proof}
\begin{remark}
The two remaining quantum symmetric algebras can be interpreted as complexifications of quantizations of the open cells of the Freudenthal variety ($(E_6, V_{\omega_1})$, $(E_6, V_{\omega_6})$, resp. the real points of the Cayley plane, which appear in the theory of the cominuscule Grassmannians.
\end{remark}

  It is interesting to observe that all these quantum symmetric algebras have multiparameter versions, which were constructed uniformly by Horton \cite{Hor} as a class of iterated skew-polynomial rings.

  In \cite{JMO} Jing, Misra and Okado introduce $q$-wedge modules, a version of quantum exterior powers for the defining representations of the classical Lie algebras.  In \cite{ZW} it is shown that the $q$-wedge modules are isomorphic to the braided or quantum exterior powers in the sense of \cite{BZ} and Definition \ref{def: cob-sym-ext-alg-pow}.  Therefore, we have the following result.

  \begin{theorem}
   The quantum exterior powers of the defining representations of the classical Lie algebras are isomorphic to the corresponding $q$-wedge modules of Jing, Misra and Okado (\cite{JMO}).
 \end{theorem}

\section{Appendix 1:Quantization of Lie bialgebras}
\label{se:quant.of LBA}
  In this section we will discuss the quantization of Lie bialgebras and show how it allows us to construct quantum symmetric algebras.

 \subsection{Quasitriangular Quasi Hopf algebras}
 In this section we will discuss the properties of quasi bi- and quasi-Hopf algebras. For a more detailed discussion see \cite[chapter 13]{ES}, whose discussion we follow closely, leaving out some of the details and proofs.
 \begin{definition}
 A quadruple $(A,\Delta,\varepsilon, \Phi)$ of an  associative algebra $A$, algebra homomorphisms $\Delta:A\to A\otimes A$, $\varepsilon:A\to k$ and an invertible element $\Phi\in A^{\otimes 3}$, called the associator, is called a quasi bialgebra if the following relations are satisfied:

\begin{enumerate}
\item $(\varepsilon\otimes 1)\Delta=(1\otimes \varepsilon)\Delta=1$.
\item $\Phi=\sum\Phi_i\otimes \Phi_j\otimes \Phi_k\in A^{\otimes 3}$ satisfies the pentagon relation:
$$\Phi_{1,2,34}\Phi_{12,3,4}=\Phi_{2,3,4}\Phi_{1,23,4}\Phi_{1,2,3}\  $$
where  $\Phi_{1,2,34}=\sum\Phi_i\otimes \Phi_j\otimes \Delta(\Phi_k)\in A^{\otimes 4}$, $\Phi_{1,2,3}=\sum\Phi_i\otimes \Phi_j\otimes \Phi_k\otimes 1\in A^{\otimes 4}$ and where $\Phi_{12,3,4},\Phi_{2,3,4}$ and $\Phi_{1,23,4}$ are defined analogously,

\item $\Delta$ is quasi-coassociative
$$\Phi(\Delta\otimes 1)\Delta(x)\Phi^{-1}=(1\otimes \Delta)\Delta(x)\ ,$$

\item $\Phi$ satisfies

$$(1\otimes \varepsilon\otimes 1)\Phi=1\ .$$
\end{enumerate}

 \end{definition}

 We have the following fact.

 \begin{proposition} (see e.g. \cite[Proposition 13.1]{ES}
 Let $(A,\Delta,\varepsilon, \Phi)$ be a quasi-bialgebra. Then the category $Rep(A)$ of $A$-modules with tensor product

 $$\pi_{V\otimes W}=(\pi_V\otimes \pi_W)\Delta\ ,$$
 the unit object $\CC$ with $\pi_\CC=\varepsilon$ and associativity isomorphisms

 $$\Phi_{X,Y,Z}:(X\otimes Y)\otimes Z\to X\otimes (Y\otimes Z)$$

     is a monoidal category.
 \end{proposition}

  We also need the notions of  quasi-Hopf algebras and quasitriangular quasi-Hopf algebras.
  Recall first the definition and some basic properties of a quasitriangular Hopf algebras.

 \begin{definition}
  A quasitriangular Hopf algebra is a pair $(H,R)$ of a Hopf $H$ and an invertible element $R\in H\otimes H$ satisfying
   \begin{equation}
\label{eq:qtr}
\Delta^{op}=R\Delta R^{-1}\ ,
\end{equation}  and the hexagon relations
\begin{equation}
\label{eq: hexagon}
(1\otimes \Delta)R=R_{13}R_{12}\ , (\Delta\otimes 1)R=R_{13}R_{23}\ ,
\end{equation}
where $R_{13}=\sum R'\otimes 1\otimes R''$ for $R= \sum R'\otimes R''$.
\end{definition}

The element $R$ is called the {\it universal R-matrix}  of $H$. Moreover if $R$ is {\it unitary}, i.e., $RR_{21}=1$, then $(H,R)$ is called a {\it triangular Hopf algebra}. The following fact is well known.

\begin{proposition}(see e.g. \cite[Proposition 9.3]{ES}
Let $H$ be a Hopf algebra and $R\in H\otimes H$ a classical $R$-matrix. Then $R$ satisfies the quantum Yang Baxter equation
\begin{equation}
\label{eq:qybe} R_{12}R_{13}R_{23}=R_{23}R_{13}R_{12}\ .
\end{equation}
\end{proposition}

  \begin{definition}

\noindent(a)  A quasi-Hopf algebra is a quasi-bialgebra $(A,\Delta,\varepsilon, \Phi)$ equipped with an antihomomorphism $S:A\to A$ and elements $\alpha,\beta\in A$ such that for all $a\in A$:

  $$m(S\otimes \alpha)\Delta(a)=\varepsilon(a)\alpha\ ,\quad m(1\otimes \beta S)\Delta(a)=\varepsilon(a)\beta\ ,$$
  $$m(S\otimes \alpha\otimes\beta S)\Phi=1\ ,\quad m(1\otimes \beta S\otimes \alpha)\Phi^{-1}=1\ .$$

\noindent(b) A quasitriangular quasi-Hopf algebra is a pair $(H,R)$ of a quasi-Hopf algebra $H$ and an invertible element $R\in H\otimes H$ such that
 \begin{equation}
\Delta^{op}=R\Delta R^{-1}\ ,
\end{equation}
$$(1\otimes \Delta)R=\Phi_{231}^{-1}R_{13}\Phi_{213}R_{12}\Phi_{123}^{-1}\ ,$$
$$ (\Delta\otimes 1)R=\Phi_{312}R_{13}\Phi_{132}^{-1}R_{23}\Phi_{123}\ ,$$

where $\Phi$ is the associator of $H$.

  \end{definition}

It will be important for our discussion to establish the equivalence of certain quasi-bialgebras.

\begin{definition}
An equivalence of quasi-bialgebras between two quasi-bialgebras $(A,\Delta,\varepsilon, \Phi)$ and $(A',\Delta',\varepsilon', \Phi')$ is a pair $(\Theta,J)$ where $\Theta:A\to A'$ is an isomorphism of algebras and $J\in A'\otimes A'$ is an invertible element, called the twist such that
$$(\varepsilon'\otimes 1)J=(1\otimes \varepsilon)J=1\ ,$$
$$\Delta'=J^{-1}(\Theta\otimes \Theta)\Delta J\ ,$$
$$\Phi'=J_{2,3} J_{1,23}\Phi J_{12,3}^{-1} J_{1,2}^{-1}\ .$$
\end{definition}

The following fact is well known.

\begin{proposition}\cite[Proposition 13.3]{ES}
\label{pr:equivofqba}
If two quasi-bialgebras are equivalent, then the associated monoidal categories are tensor equivalent
\end{proposition}

Indeed if $(\Theta,J)$ is an equivalence of quasi-bialgebras $(A,\Delta,\varepsilon, \Phi)$ and $(A',\Delta',\varepsilon', \Phi')$, then $\Theta$ induces an equivalence of categories $\mathcal{F}:Rep(A)\to Rep(A')$ and

$$ J_{V,W}:F(X)\otimes F(Y)\to F(X\otimes Y)\ ,$$
$$ x\otimes y\mapsto J(x\otimes y)$$
induces an isomorphism of the tensor product functors for $Rep(A)$ and $Rep(A')$.

We next introduce   quasi-associative module algebras. Note that our definition of  quasi-associative algebras differs from the usual notion.

\begin{definition}
Let  $(A,\Delta,\varepsilon, \Phi)$ be a quasi-bialgebra.
A  quasi-associative $(A,\Delta,\varepsilon, \Phi)$-module algebra
is a quasi-bialgebra module  together with a map
$\mu:  M\otimes M\to A$ called the multiplication, satisfying the following relations:

$$
 \begin{CD}
(M\otimes M) \otimes M@>(Id\otimes \mu)\Phi>>M\otimes M  \\
@ V(\mu\otimes Id) VV @VV\mu V\\
M \otimes M@>\mu>>  M
\end{CD}$$

Moreover, the $A$ action satisfies
$$a.(\mu(m_1 \otimes m_2)=\mu(\Delta(a)(m_1\otimes m_2)\ .$$
\end{definition}

Note the following obvious fact.

\begin{lemma}
\label{le:bi-ass}
Let $(A,\Delta, \varepsilon)$ be a bialgebra; i.e., the associator satisfies $\Phi=1^{\otimes 3}\in A^{\otimes 3}$. Then every $A$-module algebra is associative.
\end{lemma}
Denote by $Rep(A)_{alg}$ the category of  quasi-associative  module algebras over the quasi-bialgebra $(A,\Delta,\varepsilon, \Phi)$. We have the following fact.

\begin{proposition}
\label{pr:equiv-qamas}
Let $(A,\Delta,\varepsilon, \Phi)$ and $(A',\Delta',\varepsilon', \Phi')$ be equivalent quasi-bialgebras. Then there exists and equivalence of categories between $Rep(A)_{alg}$ and $Rep(A')_{alg}$.
\end{proposition}

\begin{proof}
The categories $Rep(A)$ and $Rep(A')$ of $(A,\Delta,\varepsilon, \Phi)$ (resp. $(A',\Delta',\varepsilon', \Phi')$)-modules, are equivalent as monoidal categories by Proposition \ref{pr:equivofqba}. The equivalence of  $Rep(A) _{alg}$ and $Rep(A') _{alg}$ now follows from functoriality.
\end{proof}
 \subsection{Quantizations of bi-,Hopf- and Lie bialgebras}

 In this section we will discuss the definitions of quantizations for various algebraic objects. First recall that a {\it topological}  algebra, bialgebra or Hopf algebra is an algebra, bialgebra or Hopf algebra in the category of topological spaces, which means that multiplication and, where appropriate, comultiplication, unit, counit and antipode are continuous maps with respect to the topology.

 We now define a topology on unital algebras over $k[[h]]$ following \cite[ch 3.1]{EK}. Let $A$ be a unital algebra over $k[[h]]$ and let $I\subset A$ be a proper two-sided ideal  in $A$ containing $h\in A$.
The ideal $I$ defines a translation invariant topology on $A$ for which the ideals $\{I^n:n\ge0\}$ form a basis of neighborhoods of $0$. We say that $A$ is {\it topological} if it is complete in this topology and $A/h^n A$ is a free $k[h]/(h)^n$-module for each $n\ge1$.

We now have to define the tensor product of two topological algebras.
Let $A$ and $B$ be two topological algebras and $I,J$ the corresponding ideals. Define $A\otimes B$ to be the projective limit of algebras $A/I^n\otimes_{k[h]/h^n} B/J^n$ for $n\to \infty$. The  ideal $I\otimes B+A\otimes J$ gives $A\otimes B$ the structure of a topological algebra.

From now on, the terms topological bi-or Hopf algebras will denote  bi-or Hopf algebras  over $k[[h]]$ which are topological with respect to the topology defined above.

We now make the following definition due originally to Drinfeld.
\begin{definition}(see e.g. \cite[ch.3]{Ch-P}, \cite[ch.8.5]{ES})
\noindent (a) A quantization of a co-Poisson bialgebra $(A,\delta)$ (see Definition \ref{def:co-decoratedstuff}) is a topological bialgebra $A_h$   where $A_h/h A_h\cong A$ as a bialgebra and for $x_o \equiv x(mod\ h)$
$$ \delta(x_o)\equiv  h^-1(\Delta(x)-\Delta^{op}(x)) (mod\ h)\ ,$$
with $\Delta^{op}$ denoting the opposite comultiplication.

\noindent(b) A quantization of a co-Poisson Hopf algebra  $H$ is a topological Hopf algebra $H_h$, such that $H_h/hH_h\cong H$ as Hopf algebras and $H_h$ is a quantization of the co-Poisson bialgebra $H$.

\noindent(c) A quantization of a Lie bialgebra $\gg$ is  a quantization of the corresponding co-Poisson Hopf algebra $(U(\gg),\delta)$.

\noindent(d) A quantization $U_h(\gg)$ of a Lie bialgebra $\gg$ is called preferred if $U_h(\gg)=U(\gg)[[h]]$ as an algebra.
\end{definition}

  Etingof and Kazhdan show in  \cite{EK} that the universal enveloping algebra  $U(\gg)$ of any Lie bialgebra can be quantized with quantization $U_h(\gg)$.  Indeed, they prove the following result.

  \begin{theorem}\cite{EK}
  Let $\gg$ be a finite-dimensional Lie bialgebra and let $D(\gg)$ be its double. Then $D(\gg)$ admits a preferred quantization $U_h(D(\gg))$ and $U_h(D(\gg))$ contains a subalgebra $U_h(\gg)$ which is a quantization of $\gg$.
  \end{theorem}

  Note that $U_h(\gg)$ is not necessarily preferred.

  \subsection{The Drinfeld category}

  In this section we establish the connection between  certain quasi Hopf algebras and quantizations of Lie bialgebras.  Let $\gg$ be a finite-dimensional Lie bialgebra and $(U(\gg)[[h]], \Delta,\varepsilon)$ the linear extension of the bialgebra structure on $U(\gg)$ to $U(\gg)[[h]]$. There exists (see e.g. \cite[ch.15]{ES} a certain invertible element $\Phi_{KZ}\in U(\gg)[[h]]^{\otimes 3}$, called the $KZ$-associator, such that  $(U(\gg)[[h]], \Delta,\varepsilon,\Phi_{KZ})$ has the structure of a quasi-bialgebra. Moreover there exist  $R\in  U(\gg)[[h]]^{\otimes 2}$ invertible and an algebra antiautomorphism $\tilde S:U(\gg)[[h]]\to U(\gg)[[h]]$ such that $A=(U(\gg)[[h]], \Delta,\varepsilon,\Phi_{KZ}, \tilde S, R)$ has the structure of a quasitriangular quasi-Hopf algebra. Moreover, $R$ is the universal $R$-matrix for the quantization $U_h(\gg)$. The following result is due to Drinfeld.

  \begin{theorem} (Drinfeld \cite{DR3}, see also \cite[Ch.16]{ES})
  \label{th:twist-equivalence}
 If $\gg$ is a Lie bialgebra such that $\gg$ is semisimple as a Lie algebra, then
the quasitriangular quasi-Hopf algebras  $A$  and $(U_h(\gg), R)$ are twist equivalent.
  \end{theorem}

  Denote by $\mathcal{D}$ the {\it Drinfeld category}. Objects of $\mathcal{D}$ are $U(\gg)$-modules and the morphisms between any two objects $V,W\in Ob(\mathcal{D})$ are given by

 $$Hom_\mathcal{D} (V,W)=Hom_{U(\gg)}(V,W)\otimes k[[h]]\ .$$
 Note that $\mathcal D$ is a full subcategory of $Rep(A)$. The following result is due to Drinfeld.

 \begin{theorem} (Drinfeld, see e.g. \cite[ch.16.5 ]{ES})
 Let $\gg$ be a semisimple Lie bialgebra. Then the category of $U_h(\gg)$-modules and the Drinfeld category are  equivalent as braided monoidal categories.
 \end{theorem}

\section{Appendix 2: The Classical Limit}
\label{se:appendix}
       All of the results in this section are either well known or proved in \cite{BZ}. For a more detailed discussion of the classical limit we refer the reader to \cite[Section 3.2]{BZ}.

We  will first introduce the notion of an {\it almost equivalence of categories}:

\begin{definition}
We say that a functor $F:\mathcal C\to \mathcal D$ is  an {\it almost equivalence} of $\mathcal C$ and $\mathcal D$ if:

\noindent (a) for any objects $c,c'$ of $\mathcal C$ an isomorphism $F(c)\cong F(c')$ in $\mathcal D$ implies that $c\cong c'$ in  ${\mathcal C}$;

\noindent (b) for any object in $d$ there exists an object $c$ in ${\mathcal C}$ such that $F(c)\cong d$ in ${\mathcal D}$.

\end{definition}

Denote by $\overline \OO_f$ the full (tensor) sub-category of $U(\gg)-Mod$, whose objects $\overline V$ are finite-dimensional $U(\gg)$-modules having a weight decomposition $\overline V=\oplus_{\mu\in P} \overline V(\mu)$.  The following fact will be the first result of this section.

\begin{proposition}\cite[Cor 3.22]{BZ}
\label{pr:almost equivalent O bar O}
The categories  $\OO_f$ and  $\overline \OO_f$ are almost equivalent. Under this almost equivalence a simple $U_q(\gg)$-module $V_\lambda$ is mapped to the simple $U(\gg)$-module $\overline V_\lambda$.
\end{proposition}
Let $V\cong \bigoplus_{i=1}^{n} V_{\lambda_{i}}  \in \OO_{f}$.  We call $\overline V \cong \bigoplus_{i=1}^{n} \overline V_{\lambda_{i}}  \in \overline \OO_{f}$ the {\it classical limit} of $V$ under the above almost equivalence.

 First, we have to introduce the notion of  $({\bf k},{\bf A})$-algebras and investigate their properties.
Let ${\bf k}$ be a field and ${\bf A}$ be a local subring of ${\bf k}$. Denote by $\mm$ the only maximal ideal in ${\bf A}$ and by $\tilde {\bf k}$ the  residue field of ${\bf A}$, i.e., $\tilde {\bf k}:={\bf A}/\mm$.

We say that an ${\bf A}$-submodule $L$ of a ${\bf k}$-vector space $V$ is an {\it ${\bf A}$-lattice} of $V$ if  $L$ is a free ${\bf A}$-module and ${\bf k}\otimes_{\bf A} L=V$, i.e., $L$ spans $V$ as a ${\bf k}$-vector space. Note that for any ${\bf k}$-vector space $V$ and any ${\bf k}$-linear basis ${\bf B}$ of $V$ the ${\bf A}$-span $L={\bf A}\cdot {\bf B}$ is an ${\bf A}$-lattice in $V$. Conversely, if $L$ is an ${\bf A}$-lattice in $V$, then any ${\bf A}$-linear basis ${\bf B}$ of $L$ is also a ${\bf k}$-linear basis of $V$.

Denote by $({\bf k},{\bf A})-Mod$ the  category whose objects are pairs $\VV=(V,L)$ of a ${\bf k}$-vector space $V$ and an ${\bf A}$-lattice $L\subset V$  of $V$; an arrow $(V,L)\to (V',L')$ is any ${\bf k}$-linear map $f:V\to V'$ such that $f(L)\subset L'$.

Clearly, $({\bf k},{\bf A})-Mod$ is an Abelian category. Moreover, $({\bf k},{\bf A})-Mod$ is ${\bf A}$-linear because each  $Hom(\UU,\VV)$ in $({\bf k},{\bf A})-Mod$ is an ${\bf A}$-module.

It can be easily verified that $({\bf k},{\bf A})-Mod$ is a symmetric tensor category (\cite[Lemma 3.14]{BZ}.
We have the following fact.
\begin{lemma}\cite[Lemma 3.12]{BZ}
\label{le:trafo of bases} The forgetful functor  $({\bf k},{\bf A})-Mod \to  {\bf k}-Mod$ given by $(V,L)\mapsto V$ is an almost equivalence of symmetric tensor categories.
\end{lemma}
Define a functor $\FF:({\bf k},{\bf A})-Mod \to \tilde {\bf k}-Mod$ by:
$$\FF(V,L)=L/\mm L\ $$
for any object
$(V,L)$ of $({\bf k},{\bf A})-Mod$ and for any morphism $f:(V,L)\to (V',L')$ we define
$\FF(f):L/\mm L\to L'/\mm L'$ to be a natural  $\tilde {\bf k}$-linear map.

\begin{lemma}\cite[Lemma 3.14]{BZ}
\label{le:specialization functor} $\FF:({\bf k},{\bf A})-Mod \to \tilde {\bf k}-Mod$ is a tensor functor and  almost equivalence.
\end{lemma}
Let $U$ be a ${\bf k}$-Hopf algebra and let $U_{\bf A}$ be a Hopf ${\bf A}$-subalgebra of $U$. This means that $\Delta(U_{\bf A})\subset U_{\bf A}\otimes_{\bf A} U_{\bf A}$ (where $U_{\bf A}\otimes_{\bf A} U_{\bf A}$ is naturally an ${\bf A}$-sub-algebra of $U\otimes_{\bf k} U$), $\varepsilon(U_{\bf A})\subset {\bf A}$, and $S(U_{\bf A})\subset U_{\bf A}$.
We will refer to the above pair  $\UU=(U,U_{\bf A})$ as to $({\bf k},{\bf A})$-{\it Hopf algebra} (please note that $U_{\bf A}$ is not necessarily a free ${\bf A}$-module, that is, $\UU$ is not necessarily a $({\bf k},{\bf A})$-module).

Given $({\bf k},{\bf A})$-Hopf algebra $\UU=(U,U_{\bf A})$, we say that an  object $\VV=(V,L)$ of $({\bf k}, {\bf A})-Mod$  is a $\UU$-module if $V$ is a $U$-module and $L$ is an $U_{\bf A}$-module.

Denote by $\UU-Mod$ the category which objects are $\UU$-modules  and arrows are those morphisms of $({\bf k}, {\bf A})$-modules which commute with the $\UU$-action.

 Clearly, for $({\bf k},{\bf A})$-Hopf algebra $\UU=(U,U_{\bf A})$ the category $\UU-Mod$ is a tensor (but not necessarily symmetric) category.

 For each $({\bf k},{\bf A})$-Hopf algebra $\UU=(U,U_{\bf A})$ we define $\overline \UU:=U_{\bf A}/\mm U_{\bf A}$. Clearly, $\overline \UU$ is a Hopf algebra over $\tilde {\bf k}={\bf A}/\mm$.

The following fact is obvious.

\begin{lemma}\cite[Lemma 3.15]{BZ}
\label{le:Hopf modules}
 In the notation of Lemma \ref{le:specialization functor}, for any   $({\bf k},{\bf A})$-Hopf algebra $\UU$ the functor $\FF$   naturally extends to a tensor functor
\begin{equation}
\label{eq:dequantization functor}
\UU-Mod\to \overline \UU-Mod \ .
\end{equation}

\end{lemma}

Now let ${\bf k}=\CC(q)$ and ${\bf A}$ be the ring of all those rational functions in $q$ which are defined at $q=1$. Clearly, ${\bf A}$ is a local PID with maximal ideal $\mm=(q-1){\bf A}$ (and, moreover, each ideal in ${\bf A}$ is of the form $\mm^n=(q-1)^n{\bf A}$). Therefore,  $\tilde {\bf k}:={\bf A}/\mm=\CC$.

Recall from Section \ref{se:appendix2} the definition of  the quantized universal enveloping algebra $U_{q}(\gg)$. Denote $h_\lambda=\frac{K_\lambda-1}{q-1}$ and let $U_{{\bf A}}(\gg)$ be the ${\bf A}$-algebra generated by all  $h_\lambda$, $\lambda\in P$  and all $E_ i, F_i$.

Denote by $\UU_{q}(\gg)$ the pair $(U_{q}(\gg), U_{{\bf A}}(\gg))$.

\begin{lemma}
\label{le:Hopf pair}
(a) The pair $\UU_{q}(\gg)=(U_{q}(\gg), U_{{\bf A}}(\gg))$ is a $({\bf k},{\bf A})$-Hopf algebra (\cite[Lemma 3.16]{BZ}).

\noindent(b) $\overline{\UU_q(\gg)}=U(\gg)$ (\cite[Lemma 3.17]{BZ}).
\end{lemma}

We have the following fact.

\begin{lemma}
\label{le:class. limit}
$\UU_{q}(\gg)$ is a $\UU_{q}(\gg)$-module algebra under the adjoint action. Moreover the classical limit of  $\UU_{q}(\gg)$, as a $U(\gg)$-module algebra is $U(\gg)$.
\end{lemma}

  The previous lemma has the following consequence.

  \begin{proposition}
  \label{pr:c-l-algebra}
  Let $(V,L)$ be a finite-dimensional $\UU_{q}(\gg)$-submodule of $\UU_{q}(\gg)$. Then the  algebra generated by $(V,L)$ is a $\UU_{q}(\gg)$-module algebra. Moreover there exists a $U(\gg)$-module algebra $\overline{A}$ such that $\overline A$ is the classical limit of $A$.
  \end{proposition}
    \begin{proof}
    Recall that a finitely generated module over a principal ideal domain is free if it is torsion free. Since  $U_{A}(\gg)$ is torsion free one easily obtains that the  $A$-algebra generated by $L$ is free as an $A$-module and can be given a $({\bf k}, A)$-module structure.  The  $\UU_{q}(\gg)$-action is given by the adjoint action. The second assertion follows from Lemma \ref{le:Hopf modules} and Lemma \ref{le:Hopf pair}. The proposition is proved.
    \end{proof}

Finally, we need the following fact.

\begin{proposition}
\label{pr:c-l--coPoisson}
 The classical limit of  $\UU_{q}(\gg)$ has naturally the structure of a co-Poisson algebra $(U(\gg),\delta)$ such that $\delta(\overline u)\equiv \frac{\Delta(u)-\Delta^{op}(u)}{q-1}(mod q-1)$ for $\overline u\equiv u (mod q-1)\in U_A(\gg)$.
  \end{proposition}

 \section{Appendix 3: The quantized universal enveloping algebra $U_q(\gg)$}
 \label{se:appendix2}
 We start with the definition of the quantized enveloping algebra
associated with a complex reductive Lie algebra $\gg$ (our standard
reference here will be \cite{Jo}). Let $\hh\subset \gg$
be a Cartan subalgebra,  $P(\gg)$ the weight lattice, as introduced above, and let $A=(a_{ij})$ be the Cartan matrix
for $\gg$. Additionally, let $(\cdot,\cdot)$ be the standard non-degenerate symmetric bilinear form on $\hh$.

The {\it quantized enveloping algebra} $U$ is a $\CC(q)$-algebra generated
by the elements $E_i$ and $F_i$ for $i \in [1,r]$, and $K_\lambda $ for $\lambda \in P(\gg)$,
subject to the following relations:
$K_\lambda  K_\mu = K_{\lambda +\mu}, \,\, K_0 = 1$
for $\lambda , \mu \in P$; $K_\lambda E_i =q^{(\alpha_i\,,\,\lambda)} E_i K_\lambda,
\,\, K_\lambda F_i =q^{-(\alpha_i\,,\,\lambda)} F_i K_\lambda
$
for $i \in [1,r]$ and $\lambda\in P$;
\begin{equation}
\label{eq:upper lower relations}
E_i,F_j-F_jE_i=\delta_{ij}\frac{K_{\alpha_i}- K_{-\alpha_i}}{q^{d_i}-q^{-d_i}}
\end{equation}
for $i,j \in [1,r]$, where  $d_i=(\alpha_i\, ,\,\alpha_i)$;
and the {\it quantum Serre relations}
\begin{equation}
\label{eq:quantum Serre relations}
\sum_{p=0}^{1-a_{ij}} (-1)^p 
E_i^{(1-a_{ij}-p)} E_j E_i^{(p)} = 0,~\sum_{p=0}^{1-a_{ij}} (-1)^p 
F_i^{(1-a_{ij}-p)} F_j F_i^{(p)} = 0
\end{equation}
for $i \neq j$, where
the notation $X_i^{(p)}$ stands for the \emph{divided power}
\begin{equation}
\label{eq:divided-power}
X_i^{(p)} = \frac{X^p}{(1)_i \cdots (p)_i}, \quad
(k)_i = \frac{q^{kd_i}-q^{-kd_i}}{q^{d_i}-q^{-d_i}} \ .
\end{equation}

%

The algebra $U$ is a $q$-deformation of the universal enveloping algebra of
the reductive Lie algebra~$\gg$, so it is commonly
denoted by $U = U_q(\gg)$.
It has a natural structure of a bialgebra with the co-multiplication $\Delta:U\to U\otimes U$
and the co-unit homomorphism  $\varepsilon:U\to \QQ(q)$
given by
\begin{equation}
\label{eq:coproduct}
\Delta(E_i)=E_i\otimes K_{-\alpha_i}+K_{\alpha_i}\otimes E_i\ ,\quad
\Delta(F_i)=F_i\otimes K_{-\alpha_i}+ K_{\alpha_i}\otimes F_i, \, \Delta(K_\lambda)=
K_\lambda\otimes K_\lambda \ ,
\end{equation}
\begin{equation}
\label{eq:counit}
\varepsilon(E_i)=\varepsilon(F_i)=0, \quad \varepsilon(K_\lambda)=1\ .
\end{equation}
In fact, $U$ is a Hopf algebra with the antipode anti-homomorphism $S: U \to U$ given by
\begin{equation}
\label{eq:antipode}
 S(E_i) = -q_i^{-1} E_i, \,\, S(F_i) = -qF_i, \,\,
S(K_\lambda) = K_{-\lambda}\ .
\end{equation}

Let $U^-$ (resp.~$U^0$; $U^+$) be the $\QQ(q)$-subalgebra of~$U$ generated by
$F_1, \dots, F_r$ (resp. by~$K_\lambda \, (\lambda\in P)$; by $E_1, \dots, E_r$).
It is well-known that $U=U^-\cdot U^0\cdot U^+$ (more precisely,
the multiplication map induces an isomorphism $U^-\otimes U^0\otimes U^+ \to U$).



We will consider the full sub-category  $\OO_{f}$ of the category
$U_q(\gg)-Mod$. The objects of $\OO_{f}$ are finite-dimensional
$U_{q}(\gg)$-modules $V^q$ having a weight decomposition
$$V^q=\oplus_{\mu\in P} V^q(\mu)\ ,$$
where each $K_\lambda$ acts on each {\it weight space} $V^q(\mu)$ by
the multiplication with $q^{(\lambda\,|\,\mu)}$  (see e.g.,
\cite{brown-goodearl}[I.6.12]). The category $\OO_{f}$ is semisimple
and the irreducible objects $V^q_\lambda$ are  generated by highest
weight spaces $V^q_\lambda(\lambda)=\CC(q)\cdot v_\lambda$, where
$\lambda$ is a {\it dominant weight}, i.e, $\lambda$ belongs to
$P^+=\{\lambda\in P:(\lambda\,|\,\alpha_i)\ge 0~ \forall ~i\in
[1,r]\}$, the monoid of dominant weights.

\end{document}